\documentclass[journal]{IEEEtran}
\usepackage{cite}
\usepackage{color}
\usepackage{graphicx}
\graphicspath{{fig/}{jpeg/}}
\usepackage[cmex10]{amsmath}
\usepackage{amssymb,bbm,}
\usepackage{algorithm,algorithmic,multirow,hhline}
\usepackage{booktabs}
\usepackage{amsmath,amssymb,lipsum}
\usepackage{hyperref}
\usepackage{mathrsfs}
\usepackage{comment}
\usepackage{graphicx}
\usepackage{float}
\usepackage[font=small]{caption}
\usepackage{subcaption}
\usepackage[dvipsnames]{xcolor}
\let\labelindent\relax
\usepackage{mathtools,enumitem}
\usepackage[outdir=./]{epstopdf}
\input{mysymbol.sty}
\usepackage{needspace}

\usepackage{theorem}

\usepackage{pifont}
\newcommand{\cmark}{\text{\ding{51}}}%
\newcommand{\xmark}{\text{\ding{55}}}

\newtheorem{thm}{Theorem}
\newtheorem{lemma}{Lemma}

\newtheorem{corollary}{Corollary}

\theoremstyle{definition}
\newtheorem{assumption}{Assumption}
\newtheorem{remark}{Remark}

\date{}

\def \QED {$\blacksquare$}
\newcommand{\closure}[2][3]{{}\mkern#1mu\overline{\mkern-#1mu#2}}

 \newcommand{\INDSTATE}[1][1]{\STATE\hspace{3mm}}
\newcommand{\INDSTATED}[1][1]{\STATE\hspace{6mm}}

\title{Adaptive Kernel Learning\\ in Heterogeneous Networks}
\author{Hrusikesha Pradhan, Amrit~Singh~Bedi,
	Alec~Koppel,
	and~Ketan~Rajawat

	\thanks{
		H. Pradhan and K. Rajawat are with the Dept. of EE,
		Indian Institute of Technology, Kanpur 208016, India (e-mail: $\{$hpradhan,ketan@iitk.ac.in$\}$). A. S. Bedi and A. Koppel are with U.S. Army Research Laboratory, Adelphi, MD, USA. (e-mail: amritbd@iitk.ac.in; alec.e.koppel.civ@mail.mil). Part of this work appeared at Global Conference on Signal and Information Processing, Anaheim, California, USA, November $26-29$, $2018$ \cite{8646689}.  }\vspace{-0cm}}

\begin{document}
\maketitle
\begin{abstract}
We consider learning in decentralized heterogeneous networks: agents seek to minimize a convex functional that aggregates data across the network, while only having access to their local data streams.  We focus on the case where agents seek to estimate a regression \emph{function} that belongs to a reproducing kernel Hilbert space (RKHS). To incentivize coordination while respecting network heterogeneity, we impose nonlinear proximity constraints. The resulting constrained stochastic optimization problem is solved using the functional variant of stochastic primal-dual (Arrow-Hurwicz) method which yields a decentralized algorithm. In order to avoid the model complexity from growing linearly over time, we project the primal iterates onto subspaces greedily constructed from kernel evaluations of agents' local observations. The resulting scheme, dubbed Heterogeneous Adaptive Learning with Kernels (HALK), allows us, for the first time, to characterize the precise trade-off between the optimality gap, constraint violation, and the model complexity. In particular, the proposed algorithm can be tuned to achieve zero constraint violation, an optimality gap of $\ccalO(T^{-1/2}+\alpha)$ after $T$ iterations, where the number of elements retained in the dictionary is determined by $1/\alpha$.  Simulations on a correlated spatio-temporal field estimation validate our theoretical results, which are corroborated in practice for networked oceanic sensing buoys estimating temperature and salinity from depth measurements.

%

\end{abstract}


\section{Introduction}\label{sec:intro}
In decentralized optimization, each agent $i\in \ccalV$ in a network $\mathcal{G}=(\ccalV,\mathcal{E})$ has a local objective but seek to cooperate with other agents to minimize the global network objective. The agents communicate only with their neighbors for solving the global objective. This global objective is {the} sum of local convex objectives available at different nodes of the network and depends upon the locally observed information. This framework has yielded, for instance, networked controllers \cite{olfati2004consensus}, signal processing \cite{sayed2007distributed},  robotics \cite{schwager2017multi}, and communications \cite{kozick2004source}.

In this work, we focus on the case where the agents that comprise the interconnected network may be of different types, such as aerial and ground robots collaboratively gathering information \cite{schwager2017multi}, or wireless channel estimation when spatial covariates are present \cite{kozick2004source}. In such settings, local information may be distinct, but performance may still be boosted by information sharing among agents. This phenomenon may be mathematically encapsulated as convex non-linear proximity constraints. We focus on the case where each agent's objective depends on a data stream, i.e., the online case, and the observations provided to the network are heterogeneous, when agents decisions are defined not by a standard parameter vector but instead a nonlinear regression function that belongs to a reproducing kernel Hilbert space (RKHS) \cite{Kivinen2004}.

\begin{table*}[t]
\centering
\renewcommand\tabcolsep{10pt}
\begin{tabular}{|cccccc|}
\hline
Reference & {Average Sub-optimality} & {Average Constraint violation}    & Multi-agent & Function Class  & Model Complexity   \\[5pt]
\hline
\cite{yu2017online,madavan2019subgradient}    &  $\ccalO(T^{-1/2})$ 	& $\ccalO(T^{-1/2})$         & $\xmark$      & $\reals^d$ &    N.A.\\[5pt]
\cite{mahdavi2012trading,koppel2017proximity} &  $\ccalO(T^{-1/2})$ & $\ccalO(T^{-1/4})$    & $\cmark$ & $\reals^d$ &  N.A.\\[5pt]
%
%
%
\cite{mahdavi2012trading}    &  $\ccalO(T^{-1/4})$ 	& zero         & $\xmark$       & $\reals^d$ & N.A.\\[5pt]
\cite{yu2020low}    &  $\ccalO(T^{-1/2})$ 	& $\ccalO(T^{-1})$         & $\xmark$       & $\reals^d$ & N.A.\\[5pt]
\cite{jenatton2016adaptive,yuan2018online}    &  $\ccalO(T^{-1/2})$ 	& $\ccalO(T^{-1/4})$         & $\xmark$       & $\reals^d$ & N.A.\\[5pt]
\cite{bouboulis2017online,KoppelEtal18a,richards2020decentralised,xu2020coke}    &  $\ccalO(T^{-1/2})$ & $\ccalO(T^{-1/4})$    &  $\xmark$ & $\ccalH$    & finite\\[5pt]
{\textbf{This Work}}  & {$\ccalO(T^{-1/2}+\alpha)$} & {zero} & {$\cmark$} & {$\ccalH$} & {$\ccalO(\alpha^{-2p})$}\\[5pt]
\hline
\end{tabular}
\caption{{Comparison of our work with the related works in the literature.} $\reals^d$ denotes $d$-dimensional Euclidean space, $\ccalH$ denotes an RKHS, and $p$ denotes the data dimension. {Our work achieves null constraint violation on average with comparable average sub-optimality rate.}}\label{Tab1}\vspace{-5mm}
\end{table*}
Setting aside the constraints, the solution of stochastic programs, assuming no closed form exists, necessitates iterative tools. The {simplest} approach, gradient descent, requires evaluating an expectation which depends on infinitely many data realizations. {This issue may be overcome through stochastic gradient descent (SGD) \cite{Robbins1951}, which alleviates the dependence on the sample size by using stochastic gradients in lieu of true gradients, and hence is popular in large-scale learning \cite{shapiro2009lectures}}. However, its limiting properties are intrinsically tied to the statistical parameterization (decision variable) one chooses. For vector-valued problems, i.e., linear statistical models, the convergence of SGD is well-understood via convexity \cite{Boyd2004}. 

By contrast, optimization problems induced by those with richer descriptive capability (owing to universal approximation \cite{tikhomirov1991representation}), are more challenging. 
Dictionary learning  \cite{Elad2006} and deep networks \cite{haykin1994neural} trade convexity for descriptive richness, which has led to a flurry of interest in non-convex stochastic optimization \cite{jain2017non}. Generally, overcoming non-convexity requires adding noise that degrades parameter estimates \cite{pemantle1990nonconvergence}, which may then prove inoperable for online systems. 
Instead, one may preserve convexity while obtaining nonlinearity (universality) through the ``kernel trick"  \cite{slavakis2013online}. Owing to the Representer Theorem \cite{scholkopfgeneralized}, one may transform the function variable to an inner product of weights and kernel evaluations at samples. Unfortunately, the representational complexity is proportional with the sample size $N$ \cite{scholkopfgeneralized}, which for online settings $N\rightarrow\infty$. To address this issue, we employ hard-thresholding projections onto subspaces constructed greedily from the history of data observation via matching pursuit \cite{258082}, which nearly preserves global convergence \cite{POLK}. 

Now, we shift focus to multi-agent optimization. Typically, the global cost is additive across agents' individual costs. Thus decentralized schemes may be derived by constraining agents' decisions to be equal. One may solve such problems via primal-only schemes via penalty method \cite{Ram2010},  reformulating the consensus constraint in the dual domain \cite{hosseini2013online}, and primal-dual approaches \cite{lee2016distributed} which alternate primal/dual descent/ascent steps on the Lagrangian. Approximate dual methods \cite{Arrow1958}, i.e., ADMM, have also been used \cite{shi2014linear}. Beyond linear equality constraints, motivated by heterogeneous networks, only primal methods and exact primal-dual approaches are viable, since dual methods/ADMM require solving a nonlinear $\argmin$ in the inner-loop which is prohibitively costly. 
Hence, we adopt a primal-dual approach to solving the proximity-constrained problem  \cite{koppel2017proximity} over the more general RKHS setting \cite{koppel2017decentralized}, which is developed in detail in Sec. \ref{sec:prob}. To do so, we generalize RKHS primal-dual method \cite{KoppelEtal18a} to multi-agent optimization (Sec \ref{sec:algorithm}), and obtain a new collaborative learning systems methodology which we call Heterogeneous Adaptive Learning with Kernels (HALK)(Sec \ref{sec:algorithm}). {Compared to \cite{KoppelEtal18a}, several technical contributions are unique to this work:
\begin{itemize}[leftmargin=*]
\item We establish an explicit dependence of the number of dictionary elements (model complexity) that must be retained for the proposed non-parametric approach to yield solutions that are arbitrarily accurate. Thm.  \ref{thm:bound_memory_order} characterizes that for a given choice of the parameter $\alpha$, the model order is upper bounded by  $\ccalO(\alpha^{-2p})$. Such a non-asymptotic characterization of model complexity has never been reported in the literature for any of the non-parametric algorithms. 
\item Relative to existing primal-dual algorithms for constrained  stochastic optimization in RKHS \cite{KoppelEtal18a,koppel2017decentralized}, we have introduced a regularization of the dual update in terms of problem constants that permits one to obtain {average} sub-optimality rate of $\ccalO(T^{-1/2}+\alpha)$ which is similar to the tightest sub-optimality rates $\ccalO(\sqrt{T})$ available in the literature, while ensuring  feasibility {on average}, i.e., null constraint violation, in contrast to $\ccalO(T^{3/4})$ rate for existing settings (Thm. \ref{thm:convergence}). The dependence of the optimality gap on the model complexity ($\alpha$) along with the number of iterations ($T$) is established for the first time in the RKHS literature.

%
%
%

%
%
\item In Sec. \ref{sec:simul}, we demonstrate the algorithm's effectiveness for spatio-temporal correlated Gaussian random field estimation (Sec. \ref{field_est}). Moreover, we experimentally test it on a real ocean data set for monitoring ocean salinity and temperature on buoys at various depths (Sec. \ref{sec:ocean}). An extended experiment  adapts bandwidths online \cite{singh2011information}: each agent's kernel adapts to its local data, and thus outperforms settings where local hyper-parameters are fixed.
\end{itemize}
%
 %
 {\bf Discussion of Rates} Regarding the context of  Thm. \ref{thm:convergence}, we note that existing efforts to obtain strict feasibility only obtain optimality gap attenuation at $\ccalO(T^{3/4})$ \cite{mahdavi2012trading}, or obtain $\ccalO(\sqrt{T})$ sub-optimality with comparable constraint violation \cite{yu2017online,madavan2019subgradient}, with the exception of a complicated barrier methods that do not generalize to the learning settings \cite{yu2020low}. {In this work, we achieve strict feasibility on average without degrading the sub-optimality gap.} {See Table \ref{Tab1} for details.} Moreover,  \cite{jenatton2016adaptive,yuan2018online} 
  obtain a tunable tradeoff $\beta\in(0,1)$ between sub-optimality and constraint violation, which we do not consider for simplicity, and specifically considering $\beta=0.5$, we present their best result in Table \ref{Tab1}. }

\section{Problem Formulation}\label{sec:prob}
%
In supervised learning, data takes the form of input-output examples, {$(\bbx, y)$,} which are independent identically distributed (i.i.d.) realizations from a distribution $\mathbb{P}(\bbx,y)$ with support $\ccalX \times \ccalY$. Here, $\ccalX \subset \reals^p$ and $\ccalY \subset \reals$. In this work, we focus on expected risk minimization where one seeks to compute the minimizer of a loss quantifying the merit of a nonlinear statistical model $f\in \ccalH$ averaged over the data $\{(\bbx, y)\}$. Setting aside the choice of $\ccalH$ for now, the merit of estimator $f$ is quantified by the convex loss function $\ell:\ccalH \times \ccalX \times \ccalY \rightarrow \reals$ which is small when estimator $f(\bbx)$ evaluated at feature vector $\bbx$ is close to target variable $y$. Averaged over the unknown data distribution, we obtain the statistical loss  $\mbE_{\bbx, \bby}[ \ell(f(\bbx), y)]$. It is commonplace to add a Tikhonov regularizer, yielding the regularized loss $\mbE_{\bbx, \bby}[ \ell(f(\bbx), y)] + (\lambda/2)\|f\|^2_{\ccalH}$. Within the non-parametric setting, the regularizer permits application of the Representer Theorem \cite{scholkopfgeneralized}, as discussed later in this paper.  The optimal (centralized) function is defined as
\begin{align}\label{eq:kernel_stoch_opt}
\min_{f \in \ccalH}\mbE_{\bbx, {y}}\left[ \ell\left(f(\bbx), y\right)\right] +\frac{\lambda}{2}\norm{f}^2_{\ccalH}.
\end{align}
In this work, we focus on extensions of the formulation in \eqref{eq:kernel_stoch_opt} to the case where data is scattered across an interconnected network that represents, for instance, robotic teams, communication systems, or sensor networks. To do so, we define a symmetric, connected, and directed network $\ccalG = (\ccalV, \ccalE)$ with $|\ccalV|=V$ nodes and $|\ccalE|=E$ edges. The neighborhood of agent $i$ is denoted by $n_i:=\{j:(i,j)\in\ccalE\}$.  Each agent $i\in \ccalV$ observes a local data sequence as realizations $(\bbx_{i,t}, y_{i,t})$ from random pair $(\bbx_i, y_i) \in \ccalX \times \ccalY$ and seeks to learn a optimal regression function $f_i$. {Local samples are assumed to be i.i.d. from a local distribution $(\bbx_{i,t},y_{i,t})\sim\mathbb{P}_i(\bbx_i,y_i)$ associated with node $i$.} This setting may be  encoded by associating to each node $i$ a convex loss functional $\ell_i:\ccalH \times \ccalX \times \ccalY \rightarrow \reals$ that quantifies the merit of the estimator $f_i(\bbx_i)$ evaluated at feature vector $\bbx_i$, and defining the goal for each node as the minimization of the common global loss
\vspace{-1mm}
\begin{align}\label{eq:kernel_stoch_opt_global}
\min_{ \{f_i \in \ccalH\}} S(\bbf) :=\sum_{i\in\ccalV}\Big(\mbE_{\bbx_i,y_i}\Big[ \ell_i(f_i\big(\bbx_i), y_i\big)\Big] 
+\frac{\lambda}{2}\|f_i \|^2_{\ccalH} \Big).
\end{align}
Subsequently, define the space $\ccalH^V$ whose elements are stacked functions $\bbf(\cdot) = [f_1(\cdot) ; \cdots ; f_V(\cdot)]$ that yield vectors of length $V$ when evaluated at local random vectors as $\bbf(\bbx) = [f_1(\bbx_1) ; \cdots ; f_V(\bbx_V)] \in \reals^V$. Moreover, define the stacked random vectors $\bbx = [\bbx_1 ; \cdots ; \bbx_V] \in \ccalX^V \subset \reals^{Vp}$ and $\bby = [ y_1 ; \cdots y_V] \in \reals^V$ that may represent $V$ labels or physical measurements. Note that under consensus constraints of the form $f_i = f_j$ for all $(i,j)\in \ccalE$ and for a connected network, the solutions to \eqref{eq:kernel_stoch_opt} and \eqref{eq:kernel_stoch_opt_global} coincide. 

However, as has been recently shown for the linear models, compelling all nodes to make \emph{common} decisions may ignore local differences in their data streams, and in particular, yields a sub-optimal solution with respect to their distinct data \cite{koppel2017proximity}. In other words, it is indeed possible to boost the statistical accuracy of local estimates by relaxing the consensus constraint, while still maintaining information exchange between neighbors. In the present work, we achieve this enhancement by incentivizing the agents to coordinate without forcing their estimators to completely coincide. 

To this end, we consider a convex ({symmetric}) proximity function $h_{ij}:\reals \times \reals \rightarrow \reals_{+}$ for $(i,j)\in \ccalE$. A node $i\in\ccalV$ ensures proximity to its neighbors by requiring that $f_j$ for $j\in n_i$ be such that $
H_{ij}(f_i,f_j):=\mbE_{\bbx_i} [h_{ij}(f_i(\bbx_i),f_j(\bbx_i))]\le \gamma_{ij}$
\footnote{{Due to symmetry, we assume that the proximity function is of the form $h_{ij} =\frac{1}{2} \tilde{h}_{ij}$ for a convex $\tilde{h}_{ij}$ similarly defined. Subsequently, we suppress the tilde notation for notational simplicity, which also subsumes the scaling factor of $2$ due to constraints being doubly defined.}}
  for a given tolerance $\gamma_{ij} \geq 0$. The overall stochastic network optimization problem can be written as the 
\begin{align}\label{eq:main_prob}
\!\!\!\!\!\bbf^\star=\argmin_{ \bbf \in \ccalH^{V}} S(\bbf) \ \text{s.t. } \ H_{ij}(f_i,f_j)\le \gamma_{ij}, {\forall}\  j\in n_i,
\end{align}
Observe that if $h_{ij}(f_i(\bbx_i),f_j(\bbx_i)) =| f_i(\bbx_i) - f_j(\bbx_i)|$ and $\gamma_{ij}=0$, the problem \eqref{eq:main_prob} specializes to online consensus optimization in RKHS, which was solved in an \emph{approximate} manner using penalty methods in \cite{koppel2017decentralized}. Here, we seek to obtain the \emph{exact} optimal solutions to \eqref{eq:main_prob}. In particular, we seek a solution to \eqref{eq:main_prob} whose optimality gap and constraint violation decay with the number of iterations $T$. In the subsequent section, we shift focus to doing so based upon Lagrange duality \cite{Boyd2004}. We specifically focus on distributed online settings where nodes do not know the distribution of the random pair $(\bbx_i, y_i)$ but observe local independent samples $(\bbx_{i,n}, y_{i,n})$ sequentially. 
%
\vspace{-2mm}
\subsection{Function Estimation in RKHS}\label{subsec:kernels}
The optimization problem in \eqref{eq:kernel_stoch_opt}, and hence \eqref{eq:main_prob}, is intractable in general, since it defines a variational inference problem integrated over the unknown joint distribution $\mathbb{P}(\bbx,y)$.
 However, when $\ccalH$ is equipped with a \emph{reproducing kernel} $\kappa : \ccalX \times \ccalX \rightarrow \reals$ (see \cite{slavakis2013online}), a function estimation problem of the form \eqref{eq:kernel_stoch_opt} reduces to a parametric form via the Representer Theorem \cite{kimeldorf1971some}.
Thus, we restrict $\mathcal{H}$ to be a RKHS, i.e., {for $\tilde{f}:\ccalX\rightarrow \reals$ in $\ccalH$, it holds that }\vspace{-1mm}
\begin{align} \label{eq:rkhs_properties}
& (i)  \  \langle \tilde{f} , \kappa(\bbx_i, \cdot)) \rangle _{\ccalH} = \tilde{f}(\bbx_i) ,\quad (ii)  \ \ccalH = \closure{\text{span}\{ \kappa(\bbx_i , \cdot) \}}  
\end{align}
for all $ \bbx_i \in \ccalX$. Here $\langle \cdot , \cdot \rangle_{\ccalH}$ denotes the Hilbert inner product for $\ccalH$. Further assume that the kernel is positive semidefinite, i.e. $\kappa(\bbx_i, \bbx_i') \geq 0$ for all $\bbx_i, \bbx_i' \in \ccalX$. 

%

For kernelized empirical risk minimization (ERM), {under suitable regularization,} Representer Theorem \cite{kimeldorf1971some} establishes that the optimal $\tilde{f}$ in  hypothesized function class $\ccalH$ admits an expansion in terms of kernel evaluations \emph{only} over samples
\begin{equation}\label{eq:kernel_expansion}
\tilde{f}(\bbx_i) = \sum_{k=1}^N w_{i,k} \kappa(\bbx_{i,k}, \bbx_i)\; ,
\end{equation}
where $\bbw_i = [w_{i,1}, \cdots, w_{i,N}]^T \in \reals^N$ denotes a set of weights. Here $N$ in \eqref{eq:kernel_expansion} is referred to as the model order. For ERM the model order and sample size are equal. 

{\begin{remark}({\bf Empirical Function Representation})\normalfont
Suppose, for the moment, that we have access to $N$ i.i.d. realizations of the random pairs $(\bbx_i, y_i)$ for each agent $i$ such that the expectation in \eqref{eq:main_prob} is computable, and we further ignore the proximity constraint. Defining $S^k(\bbf):=\sum_{i\in \ccalV} \ell(f_i(\bbx_{i,k}), y_{i,k}) +\frac{\lambda}{2}\|f_i \|^2_{\ccalH}$, the objective in \eqref{eq:main_prob} becomes:
\begin{align}\label{eq:kernel_batch_opt}
\min_{\bbf \in \ccalH^V} \frac{1}{N}\sum_{k=1}^N S^k(\bbf)\;. 
\end{align}
From the Representer Thm. [cf. \eqref{eq:kernel_expansion}], \eqref{eq:kernel_batch_opt} can be rewritten as
\begin{align}\label{eq:kernel_batch_opt2}
\!\!\!\!\min_{\{\bbw_i \} \in \reals^N}\frac{1}{N}\sum_{k=1}^N  \sum_{i\in\ccalV}   \ell_i(\bbw_i^T \boldsymbol{\kappa}_{\bbX_i}(\bbx_{i,k}), y_{i,k}) +\frac{\lambda}{2} \bbw_i^T \bbK_i \bbw_i  ,
\end{align}
where we have defined the kernel matrix $\bbK_i\in \reals^{N\times N}$, with entries given by the kernel evaluations between $\bbx_{i,m}$ and $\bbx_{i,n}$ as $[\bbK_i]_{m,n} =\kappa(\bbx_{i,m}, \bbx_{i,n})$.  We further define the vector of kernel evaluations $\boldsymbol{\kappa}_{\bbX_i}(\cdot) = [\kappa(\bbx_{i,1},\cdot) \ldots \kappa(\bbx_{i,N},\cdot)]^T$ related to the kernel matrix as $\bbK_i = [\boldsymbol{\kappa}_{\bbX_i}(\bbx_{i,1}) \ldots \boldsymbol{\kappa}_{\bbX_i}(\bbx_{i,N})]$, whose dictionary of associated training points is defined as  $\bbX_i = [\bbx_{i,1},\ \ldots\ ,\bbx_{i,N}]$.
The Representer Theorem allows us to transform a nonparametric infinite dimensional optimization problem in $\ccalH^V$ \eqref{eq:kernel_batch_opt} into a finite $NV$-dimensional parametric problem \eqref{eq:kernel_batch_opt2}. Thus, for ERM, the RKHS permits solving nonparametric regression problems as a search over $\reals^{VN}$ for a set of coefficients. 

However, to solve problems of the form \eqref{eq:kernel_batch_opt} when training examples $(\bbx_{i,k}, y_{i,k})$ become sequentially available or their total number $N$ is not finite, the objective in \eqref{eq:kernel_batch_opt} becomes an expectation over random pairs $(\bbx_{i}, y_{i})$ as \cite{slavakis2013online}
\begin{align}\label{eq:kernel_stoch_opt2}
{\tilde{\bbf}^\star} &=  \argmin_{\{\bbw_i\in \reals^{\ccalI}\}_{i\in\ccalV}}\sum_{i\in\ccalV}\mbE_{\bbx_i, y_i}{[ \ell_i(\sum_{n\in\ccalI} w_{i,n} \kappa(\bbx_{i,n}, \bbx_i) , y_i)}]  \nonumber \\
&\qquad \qquad \qquad+\frac{\lambda}{2} \sum_{n,m\in \ccalI } w_{i,n} w_{i,m} \kappa(\bbx_{i,m}, \bbx_{i,n})  ]  \; . 
\end{align}
The Representer Theorem{, generalized to expected value minimization, i.e., in \cite{norkin2009stochastic}, implies that the basis expansion is over a countably infinite index set $\ccalI$}. That is, as the data sample size $N\rightarrow \infty$, the representation of $f_i$ becomes infinite as well. Our goal is to solve \eqref{eq:kernel_stoch_opt2} in an approximate manner such that each {$\tilde{f}_i$} admits a finite representation near {$\tilde{f}_i^\star$}, while satisfying the proximity constraints  as mentioned in \eqref{eq:main_prob}, omitted for the sake of discussion between \eqref{eq:kernel_batch_opt} - \eqref{eq:kernel_stoch_opt2}. 
\end{remark}
}

One wrinkle in the story is that the Representer Theorem in its vanilla form \cite{kimeldorf1971some} does not apply to constrained problems \eqref{eq:main_prob}. However, recently, it has been generalized to the Lagrangian of constrained problems in RKHS \cite{KoppelEtal18a}[Thm. 1]. To this end, some preliminaries are required. Let us collect the thresholds $\gamma_{ij}$ into a vector $\gam \in \reals^{\sum_i|n_i|}$ where $|n_i|$ is the number of neighbors of node $i$. Likewise, let us compact the the functions $\{H_{ij}\}$
into a vector-valued function $\bbH:\ccalH^V \rightarrow \reals^{\sum_i |n_i|}$. Associate a non-negative Lagrange multiplier $\mu_{ij}$ for each constraint in \eqref{eq:main_prob}, and collect these Lagrange multipliers into a vector $\bbmu \in \reals^{\sum_i |n_i|}$. The Lagrangian of \eqref{eq:main_prob} is therefore given by $
\tilde\ccalL(\bbf,\bbmu) = S(\bbf) + \ip{\bbmu, \bbH(\bbf)-\gam}$. Throughout, we assume Slater's condition, implying strong duality \cite{nedic2009subgradient}. Hence the optimal of \eqref{eq:main_prob} is identical to that of the primal-dual optimal pair $(\bbf^\star,\bbmu^\star)$ of the saddle-point problem
\begin{align}\label{eq:primaldualprob}
(\bbf^\star,\bbmu^\star)=\text{arg} \hspace{0.1em} \max_{\bbmu}\hspace{0.1em}\min_{\bbf}\hspace{0.3em}\tilde{\ccalL}(\bbf,\mathbf{\bbmu}).
\end{align}


{\bf \noindent Function Representation} Now, we  establish the Representer Theorem for applying to a variant of the Lagrangian \eqref{eq:primaldualprob}. Consider the empirical approximation of $\tilde\ccalL(\bbf,\bbmu)$ with training set of node $i$ defined as $\ccalS_i=\{(\bbx_{i,1},\bby_{i,1}),\dots,(\bbx_{i,N},\bby_{i,N})\}$ over $N$ samples. The empirical version of \eqref{eq:primaldualprob} over samples $\ccalS:=\{\ccalS_1,\dots,\ccalS_{V}\}$ is written as: 
\begin{align}\label{eq:primaldualprob_emp}
(\bbf^\star,\bbmu^\star)=\text{arg} \hspace{0.2em} \max_{\bbmu}\hspace{0.2em}\min_{\bbf}\hspace{0.3em}\ccalL^e(\bbf,\mathbf{\bbmu}),
\end{align}
%
%
where $\ccalL^e(\bbf,\bbmu)$ the empirical form of the Lagrangian:\vspace{-1mm}
\begin{align}\label{eq:empirical_lagrangian}
&\ccalL^e(\bbf,\bbmu)\coloneqq \frac{1}{N}\sum_{k=1}^N \left[S^k(\bbf)+\ip{\bbmu,\bbH^k(\bbf)-\gam}\right]
\end{align}
where $\bbH^k(\bbf)$ collects the functions $H_{ij}^k(f_i,f_j) := h_{ij}(f_i(\bbx_{i,k}),f_j(\bbx_{i,k}))$  over all $i\in\ccalV$ and $j \in n_i$. Now, with this empirical formulation, we generalize the Representer Theorem for constrained settings to the multi-agent problem \eqref{eq:empirical_lagrangian} as a corollary of \cite{KoppelEtal18a}[Thm. 1] with the proof provided in Appendix \ref{proof_representthm} of the supplementary material.
\begin{corollary}\label{thm:representer}
Let $\ccalH$ be a RKHS equipped with kernel $\kappa$ and $\ccalS$ be the training data of the network. Each function $i$ that is a primal minimizer of  \eqref{eq:empirical_lagrangian} takes the form
%
$f_i^\star=\sum_{k=1}^N w_{i,k} \kappa(\bbx_{i,k},.)$
%
where $w_{i,k}\in\mbR$ are coefficients.
\end{corollary}
Next, we shift to solving \eqref{eq:main_prob} in distributed online settings where nodes do not know the distribution of the random pair $(\bbx_i, y_i)$ but observe local samples $(\bbx_{i,k}, y_{i,k})$ sequentially, through use of the Representer Theorem as stated in Corollary \ref{thm:representer} that makes the function parameterization computable.

{With the Representer theorem in place as explained above, we are ready to present the conservative version of \eqref{eq:main_prob} to vanish the long term constraint violation. To do so, we add $\nu$ to the constraint in \eqref{eq:main_prob} to reformulate the problem as follows 
\begin{align}\label{eq:prob_zero_cons}
\!\!\!\bbf_\nu^\star\!=\! \argmin_{ \{f_i \in \ccalH\}}S(\bbf)
 \ \text{s.t.} \ H_{ij}(f_i,f_j\!) \!+\! \nu \!\le\! \!\gamma_{ij},~ {\forall} ~j\!\in\! n_i,
\end{align}
By modifying \eqref{eq:prob_zero_cons}, we consider a stricter constraint than \eqref{eq:main_prob}. Doing so permits us to establish that approximate algorithmic solutions to \eqref{eq:prob_zero_cons} ensure the constraints in \eqref{eq:main_prob} may be exactly satisfied. Moreover, we are able to do so while tightening existing bounds on the the sub-optimality in \cite{mahdavi2012trading} in Sec. \ref{sec:convergence}.}
\vspace{-1mm}
\section{Algorithm Development}\label{sec:algorithm}

In this section, we develop an online and decentralized algorithm for \eqref{eq:prob_zero_cons} when  $\{f_i \}_{i \in \ccalV}$ belong to a RKHS. Begin by defining the augmented Lagrangian relaxation of \eqref{eq:prob_zero_cons}:
\begin{align}\label{eq:lagrangian_dist}
\ccalL(\bbf,\bbmu) = & S(\bbf) + \ip{\bbmu,\bbH(\bbf)+\nu\one-\gam} - \tfrac{\delta\eta}{2}\norm{\bbmu}^2
\end{align}
where $\one$ denotes the all-one vector of appropriate size and the parameters $\delta ,\nu > 0$.  The dual regularization term  in \eqref{eq:lagrangian_dist} is included in the design to control the violation of non-negative constraints on the dual variable over time $t$. For future reference, we denote $\ccalL^s(\bbf,\bbmu)$ as the standard Lagrangian, which is \eqref{eq:lagrangian_dist} with $\delta=0$. We consider the stochastic approximation of \eqref{eq:lagrangian_dist} evaluated at sample $(\bbx_{i,t},y_{i,t})$,
\begin{align}\label{eq:stochastic_approx}
&\hat{\ccalL}_t(\bbf,\bbmu) \coloneqq S^t(\bbf) + \ip{\bbmu,\bbH^t(\bbf)+\nu\one-\gam} - \tfrac{\delta\eta}{2}\norm{\bbmu}^2
\end{align}
%
With this definition, we propose applying primal-dual method to \eqref{eq:stochastic_approx} -- see \cite{Arrow1958}. To do so, we first require the functional stochastic gradient of \eqref{eq:stochastic_approx} evaluated at a sample point $(\bbx_t, \bby_t)$. Begin by considering the local loss term in \eqref{eq:stochastic_approx} :
\begin{align}\label{eq:stochastic_grad}
  \nabla_{f_i} \ell_i(f_i(\bbx_{i,t}),y_{i,t})
= \frac{\partial \ell_i(f_i(\bbx_{i,t}),y_{i,t})}{\partial f_i(\bbx_{i,t})}\frac{\partial f_i(\bbx_{i,t})}{\partial f_i}
\end{align}
where we have applied the chain rule. Now, define the short-hand notation $\ell_i'(f_i(\bbx_{i,t}),y_{i,t}): ={\partial \ell_i(f_i(\bbx_{i,t}),y_{i,t})}/{\partial f_i(\bbx_{i,t})} $ for the derivative of $\ell_i(f(\bbx_{i,t}),y_{i,t})$ with respect to its first scalar argument $f_i(\bbx_{i,t})$ evaluated at $\bbx_{i,t}$. 

To evaluate the second term on the right-hand side of \eqref{eq:stochastic_grad}, differentiate both sides of the expression and use the reproducing property of the kernel with respect to $f_i$ to obtain
\begin{align}\label{eq:stochastic_grad2}
\frac{\partial  f_i(\bbx_{i,t})}{\partial f_i} = \frac{\partial \langle f_i , \kappa(\bbx_{i,t}, \cdot) \rangle _{\ccalH}}{\partial f_i}
= \kappa(\bbx_{i,t},\cdot).
\end{align}
%
%
%
Substitute the kernel at $\bbx_{i,t}$ on the right of \eqref{eq:stochastic_grad2} into \eqref{eq:stochastic_grad}:
\begin{align}\label{eq:lagg_derv}
\nabla_{f_i}\hat{\ccalL}_{t}(\bbf_t,\bbmu_t)&=\ell_i'(f_i(\bbx_{i,t}),y_{i,t}) ~\kappa(\bbx_{i,t},\cdot)+\lambda f_i\\
&\quad+  \sum_{j\in n_i}\mu_{ij}h'_{ij}(f_i(\bbx_{i,t}),f_j(\bbx_{i,t}))~\kappa(\bbx_{i,t},\cdot).\nonumber
\end{align}
where we apply analogous logic as that which yields \eqref{eq:stochastic_grad} to \eqref{eq:lagg_derv}. To simplify, define the $V$-fold stacking of \eqref{eq:lagg_derv} as
%
\begin{align}\label{eq:lagg_derv_stack}
\nabla_\bbf\hat{\ccalL}_{t}(\bbf_t,\bbmu_t)=\text{vec}[\nabla_{f_i}\hat{\ccalL}_{t}(\bbf_t,\bbmu_t)].
\end{align}
With these definitions, saddle point method on the augmented Lagrangian \eqref{eq:lagrangian_dist}, which operates by alternating primal/dual stochastic gradient descent/ascent steps, is given as \cite{Arrow1958}:\vspace{-1mm}
%
\begin{align}
\bbf_{t+1}=& (1-\eta\lambda)\bbf_{t}-\eta \text{vec}\Big(\Big[\ell_i'(f_{i,t}(\bbx_{i,t}),y_{i,t}) \label{eq:primalupdate_comp}
\\
&+\sum_{j\in n_i}\mu_{ij,t}h'_{ij}(f_{i,t}(\bbx_{i,t}),f_{j,t}(\bbx_{i,t}))\Big]\kappa(\bbx_{i,t},\cdot) \Big), \nonumber\\
\bbmu_{t+1}=&\Big[\bbmu_t+\eta\nabla_{\bbmu}\hat{\ccalL}_{t}(\bbf_t,\bbmu_t)\Big]_{+}\label{eq:dualupdate}.
\end{align}
%
\begin{algorithm}[t]
\caption{Heterogeneous Adaptive Learning with Kernels (\textbf{\textbf{HALK}})}
\begin{algorithmic}
\label{alg:soldd}
\REQUIRE $\{\bbx_t,\bby_t,\epsilon_t \}_{t=0,1,2,...}$, $\eta$, $\nu$ and $\delta$
\STATE \textbf{initialize} ${f}_{i,0}(\cdot) = 0, \bbD_{i,0} = [], \bbw_0 = []$
\FOR{$t=0,1,2,\ldots$}
	\STATE {\bf loop in parallel} for agent $ i \in \ccalV$ 
	\INDSTATE Observe local training example realization $(\bbx_{i,t}, y_{i,t})$
	\INDSTATE Send $\bbx_{i,t}$ to neighbors $j\in n_i$ and receive $f_{j,t}(\bbx_{i,t})$  
	\INDSTATE Receive $\bbx_{j,t}$ from neighbors, $j\in n_i$ and send $f_{i,t}(\bbx_{j,t})$  \vspace{-3mm}
	\INDSTATE Compute unconstrained stochastic grad. step  using \eqref{eq:primalupdate}\vspace{-3mm}
	\INDSTATE Update dual variables for $j$ $\in$ $n_i$ using \eqref{eq:dualupdate_edge}
	\INDSTATE Update params: $\tbD_{i,t+1} = [\bbD_{i,t},\;\;\bbx_{i,t}]$, $\tbw_{i,t+1} $ \eqref{eq:param_update}
		\INDSTATE Greedily compress function using matching pursuit \vspace{-1mm}
	$$({f}_{i,t+1},\bbD_{i,t+1},\bbw_{i,t+1}) = \textbf{KOMP}(\tilde{f}_{i,t+1},\tbD_{i,t+1},\tbw_{i,t+1},\epsilon)$$ \vspace{-4mm}
	\STATE {\bf end loop}
\ENDFOR
\end{algorithmic}
\end{algorithm}
{We note that local losses $\ell$ are assumed differentiable; however, the technical development easily extends to non-differentiable cases and the theoretical analysis ends up being very similar  \cite[Eq. (9)]{Kivinen}. If a unique gradient is not available is practice, $\partial$  denotes a subgradient and we choose it arbitrarily from the set of available subgradients to perform the update in \eqref{eq:primalupdate_comp}.   }Moreover, we require the step-size $\eta< 1/\lambda$ for regularizer $\lambda>0$ in \eqref{eq:kernel_stoch_opt}. Observe that \eqref{eq:primalupdate_comp} decouples by agent $i\in\ccalV$:
\vspace{-1mm}
\begin{align}\label{eq:primalupdate}
f_{i,t+1}=& f_{i,t}(1-\eta\lambda)-\eta\Big[\ell_i'(f_{i,t}(\bbx_{i,t}),y_{i,t}) 
\\
&+ \sum_{j\in n_i}\mu_{ij,t}h'_{ij}(f_{i,t}(\bbx_{i,t}),f_{j,t}(\bbx_{i,t}))\Big]\kappa(\bbx_{i,t},\cdot).\nonumber
\end{align}
Note that the dual update in \eqref{eq:dualupdate} is vector-valued, and defined for each edge $(i,j)\in\ccalE$. Since the constraints involve only pairwise interactions between nodes $i$ and neighbors $j\in n_i$, the dual update separates along each edge $(i,j)$:
\begin{align}\label{eq:dual_gradient_seperability}
\!\!\!\!\nabla_{\mu_{ij}\!}\hat{\ccalL}_{t}(\!\bbf_t,\bbmu_t)\!=\!h_{ij}(f_{i,t}(\bbx_{i,t}),f_{j,t}(\bbx_{i,t}\!)\!)\!-\!\gamma_{ij}\!+\!{\nu}\!-\!\delta\eta\mu_{ij}.\!\!\!
\end{align}
The update of $\bbmu_t$ is carried out by substituting \eqref{eq:dual_gradient_seperability} in  \eqref{eq:dualupdate} and using the fact that vector-wise projection is applied entry-wise and thus the individual local  updates $\mu_{ij}$ can be written as
\begin{align}\label{eq:dualupdate_edge}
\mu_{ij,t+1}\!\!=\!\!\Big[\mu_{ij,t}(\!1\!-\!\delta\eta^2)\!+\!\eta\!\Big(\!\!h_{ij}(f_{i,t}(\bbx_{i,t}),f_{j,t}(\bbx_{i,t}))\!-\!\gamma_{ij}\!\!+\!{\nu}\!\Big)\!\Big]_{+}.
\end{align}
The sequence of $(\bbf_t,\bbmu_t)$ is initialized by $\bbf_0=0\in\ccalH^V$ and $\bbmu=0\in \mbR^E_+$. Using the Representer theorem, $f_{i,t}$ can be written in terms of kernels evaluated at past observations as\vspace{-2mm}
%
%
\begin{align}\label{eq:kernel_expansion_t}
f_{i,t}(\bbx) 
= \sum_{n=1}^{t-1} w_{i,n} \kappa(\bbx_{i,n}, \bbx)
= \bbw_{i,t}^T\boldsymbol{\kappa}_{\bbX_{i,t}}(\bbx) \; .
\end{align}
 We define $\bbX_{i,t} = [\bbx_{i,1}, \ldots, \bbx_{i,t-1}]\in \reals^{p\times (t-1)}$,
 $\boldsymbol{\kappa}_{\bbX_{i,t}}(\cdot) = [\kappa(\bbx_{i,1},\cdot),\ \ldots\ ,\kappa(\bbx_{i,t-1},\cdot)]^T$, and $\bbw_{i,t} = [w_{i,1} , \ \ldots \, w_{i,t-1}]^T \in \reals^{t-1}$ on the right-hand side of \eqref{eq:kernel_expansion_t}.   
 Combining the update in \eqref{eq:primalupdate} along with the kernel expansion in \eqref{eq:kernel_expansion_t}, implies that the primal functional stochastic descent step in $\ccalH^V$ results in the following $V$ parallel parametric updates on both kernel dictionaries $\bbX_i$ and $\bbw_i$:\vspace{-2mm}
\begin{align}\label{eq:param_update} 
\bbX_{i,t+1} =& [\bbX_{i,t}, \;\; \bbx_{i,t}] \; ,  \\
 [\bbw_{i,t+1}]_u  =&  
\begin{cases}
    (1 - \eta \lambda) [\bbw_{i,t}]_u \qquad\qquad\qquad \text{for } 0 \leq u \leq t-1 \\
    -\eta\Big(\ell_i'(f_{i,t}(\bbx_{i,t}),y_{i,t}) +\sum_{j\in n_i}\mu_{ij,t}h'_{ij} \\
\qquad\times (f_{i,t}(\bbx_{i,t}),f_{j,t}(\bbx_{i,t}))\Big),\text{for}~u=t
  \end{cases} \nonumber 
\end{align}
%
 From \eqref{eq:param_update} we note that each time one more column gets added to the columns in $\bbX_{i,t}$, an instance of the curse of kernelization \cite{wang2012breaking}. We define the number of data points, i.e., the number of columns of $\bbX_{i,t}$ at time $t$ as the \emph{model order}. We note that for the update in  \eqref{eq:primalupdate}, the model order is $t-1$ and it grows unbounded with iteration index $t$. This challenge often appears in connecting nonparametric statistics and optimization methods \cite{slavakis2013online}. Next, motivated by \cite{POLK}, we now define compressive subspace projections of the function sequence defined by \eqref{eq:primalupdate_comp} to trade off memory and optimality.

\vspace{2mm}
%
{\bf \noindent Complexity Control via Subspace Projections}
%
 To alleviate the aforementioned memory bottleneck, we project the function sequence  \eqref{eq:primalupdate} onto a lower dimensional subspace such that $\ccalH_\bbD \subseteq \ccalH$, where $\ccalH_\bbD$ is represented by a dictionary  $\bbD = [\bbd_1,\ \ldots,\ \bbd_M] \in \reals^{p \times M}$. Being specific, $\ccalH_\bbD$ has the form $\ccalH_\bbD = \{f\ :\ f(\cdot) = \sum_{n=1}^M w_n\kappa(\bbd_n,\cdot) = \bbw^T\boldsymbol{\kappa}_{\bbD}(\cdot) \}=\text{span}\{\kappa(\bbd_n, \cdot) \}_{n=1}^M$, and $\{\bbd_n\} \subset \{\bbx_u\}_{u\leq t}$. For convenience we define $\boldsymbol{\kappa}_{\bbD}(\cdot)=[\kappa(\bbd_1,\cdot), \ldots,  \kappa(\bbd_M,\cdot)]^T$, and $\bbK_{\bbD,\bbD}$ as the resulting kernel matrix from this dictionary. In a similar manner, we define dictinaries $\bbD_{i,t}$ and subspace $\ccalH_{\bbD_{i,t}}$  for each agent at time $t$. Similarly, the model order (i.e., the number of columns) of the dictionary $\bbD_{i,t}$ is denoted by $M_{i,t}$. We enforce function parsimony by selecting dictionaries $\bbD_i$ with $M_{i,t} << \ccalO(t)$ for each $i$ \cite{POLK}.

Now, we propose projecting the update in \eqref{eq:primalupdate} to a lower dimensional subspace $\ccalH_{\bbD_{i,t+1}}=\text{span}\{ \kappa(\bbd_{i,n}, \cdot) \}_{n=1}^{M_{t+1}}$  as
\begin{align}\label{eq:projection_hat}
{f}_{i,t+1}=& \argmin_{f \in \ccalH_{\bbD_{i,t+1}}} \| f -(f_{i,t}  - \eta \nabla_{f_i}\hat{\ccalL}_{t}(\bbf_t,\bbmu_t)) 
\|_{\ccalH}^2  \\
:=&\ccalP_{\ccalH_{\bbD_{i,t+1}}} \Big[ 
f_{i,t}(1-\eta\lambda)-\eta\Big[\ell_i'(f_{i,t}(\bbx_{i,t}),y_{i,t})\nonumber\\
&+ \sum_{j\in n_i}\mu_{ij,t}h'_{ij}(f_{i,t}(\bbx_{i,t}),f_{j,t}(\bbx_{i,t}))\Big]\kappa(\bbx_{i,t},\cdot)\Big]  \nonumber
\end{align}
where we define the projection operator $\ccalP_{\bbD_{i,t+1}}$ onto subspace $\ccalH_{\bbD_{i,t+1}}\subset \ccalH$ by the update \eqref{eq:projection_hat}.
%

\begin{algorithm}[t]
{\caption{Destructive Kernel Orthogonal Matching Pursuit (KOMP) \hspace{-2mm}}
\begin{algorithmic}\label{alg:komp}
\REQUIRE  function $\tilde{f}$ defined by dict. $\tbD \in \reals^{p \times \tilde{M}}$, coeffs. $\tbw \in \reals^{\tilde{M}}$, approx. budget  $\epsilon > 0$ \\
\STATE \textbf{initialize} $f=\tilde{f}$, dictionary $\bbD = \tbD$ with indices $\ccalI$, model order $M=\tilde{M}$, coeffs.  $\bbw = \tbw$.
\WHILE{candidate dictionary is non-empty $\ccalI \neq \emptyset$}
{\FOR {$j=1,\dots,\tilde{M}$}
	\STATE Find minimal error with basis point $\bbd_j$ removed \vspace{-2mm}
	$$\gamma_j = \min_{\bbw_{\ccalI \setminus \{j\}}\in\reals^{{M}-1}} \|\tilde{f}(\cdot) - \sum_{k \in \ccalI \setminus \{j\}} w_k \kappa(\bbd_k, \cdot) \|_{\ccalH} \; .$$ \vspace{-3mm}
\ENDFOR}
	\STATE Find minimal error index: $j^* = \argmin_{j \in \ccalI} \gamma_j$
	\INDSTATE{{\bf{if }} minimal error exceeds threshold $\gamma_{j^*}> \epsilon$}
	\INDSTATED{\bf stop} 
	\INDSTATE{\bf else} 
	
	\INDSTATED Prune dictionary $\bbD\leftarrow\bbD_{\ccalI \setminus \{j^*\}}$
	\INDSTATED Revise set $\ccalI \leftarrow \ccalI \setminus \{j^*\}$, model order ${M} \leftarrow {M}-1$.\vspace{-4mm}
	\INDSTATED Compute weights $\bbw$ defined by current dict. $\bbD$
 	\vspace{-2mm}$$\bbw = \argmin_{\bbw \in \reals^{{M}}} \lVert \tilde{f}(\cdot) - \bbw^T\boldsymbol{\kappa}_{\bbD}(\cdot) \rVert_{\ccalH}$$
 	\vspace{-5mm}
	\INDSTATE {\bf end}
\ENDWHILE	
\RETURN ${f},\bbD,\bbw$ of model order $M \leq \tilde{M}$ s.t. $\|f - \tilde{f} \|_{\ccalH}\leq \eps$
\end{algorithmic}}
\end{algorithm}

\smallskip
{\bf \noindent Coefficient update.} 
The update \eqref{eq:projection_hat} may be rewritten in terms of updates for the dictionary $\tbD_{i,t+1}$ and coefficients $\tbw_{i,t+1}$:
\begin{align}\label{eq:param_tilde} 
&\tbD_{i,t+1}  = [\bbD_{i,t},\;\;\bbx_{i,t}]\; ,  \\
& [\tbw_{i,t+1}]_u   = 
\begin{cases}
    (1 - \eta \lambda) [\bbw_{i,t}]_u, \quad \text{for } {0 \leq u \leq M_t} \\
       [\bbw_{i,t+1}]_u \ \text{from}\ \eqref{eq:param_update} \text{for } \  u=M_{t}+1\nonumber
  \end{cases} 
\end{align}
from the un-projected functional update step in \eqref{eq:primalupdate}. We denote the un-projected functional update as $\tilde{f}_{i,t+1} ={f}_{i,t} -\eta \nabla_{f_i}\hat{\ccalL}_{t}(\bbf_t,\bbmu_t)$ whose stacked functional version, using the stacked functional stochastic gradient \eqref{eq:lagg_derv_stack}, can be written as
\begin{align}\label{eq:stacked_sgd_tilde}
\tilde{\bbf}_{t+1} ={\bbf}_{t} -\eta \nabla_{\bbf}\hat{\ccalL}_{t}(\bbf_t,\bbmu_t) \; .
\end{align}
%
%
The number of columns of dictionary $\tbD_{i,t+1}$ is $M_{i,t} + 1$, which is also the length of $\tbw_{i,t+1}$. For now, to simplify notation, we denote {$\tilde{M}_{i,t+1}:=M_{i,t} + 1$}. For a given dictionary $\bbD_{i,t+1}$, projecting $\tilde{f}_{i,t+1}$ onto the subspace $\ccalH_{\bbD_{i,t+1}}$ is equivalent to finding coefficients $\bbw_{i,t+1}$ w.r.t. dictionary $\bbD_{i,t+1}$ given as
\begin{equation} \label{eq:hatparam_update}
\bbw_{i,t+1}=  \bbK_{\bbD_{i,t+1} \bbD_{i,t+1}}^{-1} \bbK_{\bbD_{i,t+1} \tbD_{i,t+1}} \tbw_{i,t+1} \;,
\end{equation}
where we define the cross-kernel matrix $\bbK_{\bbD_{i,t+1},\tbD_{i,t+1}}$ whose $(n,m)^\text{th}$ entry is given by $\kappa(\bbd_{i,n},\tbd_{i,m})$. The other kernel matrices $\bbK_{\tbD_{i,t+1},\tbD_{i,t+1}}$ and $\bbK_{\bbD_{i,t+1},\bbD_{i,t+1}}$ are defined similarly.  The number of columns in $\bbD_{i,t+1}$ is ${M}_{i,t+1}$, while the number of columns in $\tbD_{t+1}$ [cf. \eqref{eq:param_tilde}] is $\tilde{M}_{i,t+1}=M_{i,t} + 1$. Next we see, how the dictionary $\bbD_{i,t+1}$ is obtained from  $\tbD_{i,t+1}$.

\smallskip
{\bf \noindent Dictionary Update.} 
%
The dictionary $\bbD_{i,t+1}$ is selected based upon greedy compression \cite{258082}, i.e., $\bbD_{i,t+1}$ is formed from $\tbD_{i,t+1}$ by selecting a subset of  $M_{i,t+1}$ columns from $\tilde{M}_{i,t+1}$ number of columns of $\tbD_{i,t+1}$  that best approximate $\tilde{f}_{i,t+1}$ in terms of Hilbert norm error, i.e., $\|f_{i,t+1} - \tilde{f}_{i,t+1} \|_{\ccalH} \leq \eps $, where $\eps$ is the error tolerance, which may be done by \emph{kernel orthogonal matching pursuit} (KOMP) \cite{Vincent2002} written as $(f_{i,t+1}, \bbD_{i,t+1}, \bbw_{i,t+1})= \text{KOMP}(\tilde{f}_{i,t+1},\tilde{\bbD}_{i,t+1}, \tilde{\bbw}_{i,t+1},\eps)$.

We use a destructive variant of KOMP with pre-fitting as done in \cite{POLK} {and present it in Algorithm \ref{alg:komp}}. This algorithm starts with the full dictionary  $\tbD_{i,t+1}$ and sequentially removes the dictionary elements till the  condition $\|f_{i,t+1} - \tilde{f}_{i,t+1} \|_{\ccalH} \leq \eps $ is violated. In order to ensure the boundedness of the primal iterates in the subsequent section, we consider a variant of KOMP that explicitly enforces the projection to be contained within a finite Hilbert norm ball, which has the practical effect of thresholding the coefficient vector if it climbs above a certain large finite constant. Next we delve into the convergence of updates \eqref{eq:projection_hat} and \eqref{eq:dualupdate_edge}, summarized as Algorithm \ref{alg:soldd}.

\vspace{-1mm}
\section{Convergence Analysis}\label{sec:convergence}

In this section, we establish the convergence of Algorithm \ref{alg:soldd} by characterizing   both objective sub-optimality and constraint violation in expectation. Before doing so, we define terms to clarify the analysis. Specifically, the projected functional stochastic gradient associated with  \eqref{eq:projection_hat} is defined as
\begin{align}\label{eq:proj_stochasticgrad}
\!\!\!\!\!\!\!\tilde{\nabla}_{\!\!f_i}\hat{\ccalL}_{t}(\!\bbf_{t},\bbmu_{t})\!=\!\!\big(f_{i,t}\!-\!\!\ccalP_{\ccalH_{\bbD_{i,t+1}}}\![f_{i,t}\!-\!\eta \nabla_{f_i}\hat{\ccalL}_{t}(\bbf_{t},{\bbmu}_{t})]\big)\!/\eta.\!\!\!
\end{align}
Using \eqref{eq:proj_stochasticgrad}, the update \eqref{eq:projection_hat} can be rewritten as $f_{i,t+1}=f_{i,t}-\eta\tilde{\nabla}_{f_i}\hat{\ccalL}_{t}(\bbf_{t},\bbmu_{t})$. We stack the projected stochastic functional gradient $\tilde{\nabla}_{f_i}\hat{\ccalL}_{t}(\bbf_{t},\bbmu_{t})$ and define the stacked version \\$\tilde{\nabla}_{\bbf}\hat{\ccalL}_{t}(\bbf_{t},\bbmu_{t})=[\tilde{\nabla}_{f_1}\hat{\ccalL}_{t}(\bbf_{t},\bbmu_{t}),\dots,\tilde{\nabla}_{f_V}\hat{\ccalL}_{t}(\bbf_{t},\bbmu_{t})]$. Using this stacked gradient, the update \eqref{eq:projection_hat} then takes the form
\begin{align}\label{eq:projected_func_update}
\bbf_{t+1}=\bbf_t-\eta \tilde{\nabla}_{\bbf}\hat{\ccalL}_{t}(\bbf_{t},\bbmu_{t}).
\end{align}
Next, we state the assumptions required for the convergence.\vspace{-1mm}
\begin{assumption}\label{as:first}
The feature space $\ccalX\subset\reals^p$ and target domain $\ccalY\subset\reals$ are compact, and the  kernel map may be bounded as\vspace{-1mm}
\begin{equation}\label{eq:bounded_kernel}
\sup_{\bbx\in\ccalX} \sqrt{\kappa(\bbx, \bbx )} = X < \infty
\end{equation}
\end{assumption}\vspace{-5mm}
\begin{assumption}\label{as:second}
The local losses $\ell_i(f_i(\bbx), y)$ {are convex and non-differentiable} with respect to the first (scalar) argument $f_i(\bbx)$ on $\reals$ for all $\bbx\in\ccalX$ and $y\in\ccalY$.  
Moreover, the instantaneous losses $\ell_i: \ccalH \times \ccalX \times \ccalY \rightarrow \reals$ are $C_i$-Lipschitz continuous
\begin{equation}\label{eq:lipschitz}
| \ell_i(z, y) - \ell_i( z', y ) | \leq C_i |z - z'|  \ \text{ for all $z  $ with $y$ fixed}. 
\end{equation}
Further denote $C :=\max_i C_i$ (largest modulus of continuity).
\end{assumption}
\begin{assumption}\label{as:third}
The constraint functions $h_{ij}$ for all $(i,j)\in \ccalE$ are all uniformly $L_h$-Lipschitz continuous in its first (scalar) argument; i.e., for any $z$, $z'\in \reals$, there exist constant $L_h$, s.t.
\begin{equation}
|h_{ij}(z,y)-h_{ij}(z',y)|\le L_h|z-z'|
\end{equation}
and is also convex w.r.t the first argument $z$.
\end{assumption}
\begin{assumption}\label{as:fourth}
There exists $\bbf^{\dagger}$ such that for all $(i,j)\in\ccalE$, we have $h_{ij}(f^{\dagger}_i,f^{\dagger}_j)+\xi\le \gamma_{ij}$, for some $\xi>0$, which implies that the constraint is strictly satisfied.
\end{assumption}

\begin{assumption}\label{as:fifth}
The functions $f_{i,t+1}$ output from KOMP have Hilbert norm bounded by $R_{\ccalB}< \infty$. Also, the optimal $f_i^\star$ lies in the ball $\ccalB$ with radius $R_{\ccalB}$.
\end{assumption}
%
%
%
Often,  Assumption \ref{as:first} holds by  data domain itself. Assumptions \ref{as:second} and \ref{as:third}  ensure that the constrained stochastic optimization problem is { convex and the objective as well as constraint function satisfies the Lipschitz property, which are typical of first-order methods \cite{Kivinen}. Moreover, the analysis in this work easily extends to the case where $\ell$ is non-differentiable.} Assumption \ref{as:fourth}, i.e., Slater's condition, ensures the feasible set of \eqref{eq:main_prob} is non-empty, and is standard in primal-dual methods \cite{nedic2009subgradient}. Assumption \ref{as:fifth} ensures that the algorithm iterates and the optimizer are finite, and their domains overlap. It may be explicitly enforced by dividing the norm of the coefficient vector output from KOMP by a large constant.

We begin with establishing an upper bound on the memory order of function $f_{i,t}$ obtained from Algorithm \ref{alg:soldd}.  Hence, using   Assumption \ref{as:second} and \ref{as:third} we present the model order theorem. 
\begin{thm}\label{thm:bound_memory_order}
Let $f_{i,t}$ denote the function sequence of agent $i$ at $t$th instant generated from Algorithm \ref{alg:soldd} with dictionary $\bbD_{i,t}$. Denote $M_{i,t}$ as the model order representing the number of dictionary elements in $\bbD_{i,t}$. Then with constant step size $\eta$ and compression budget $\eps$, for a Lipschitz Mercer kernel $\kappa$ on a compact set $\ccalX \subset \reals^p$, there exists a constant $\beta$ such that for any training set $\{\bbx_{i,t}\}_{t=1}^\infty$,  $M_{i,t}$ satisfies\vspace{-2mm}
\begin{align}\label{eq:agent_mo}
M_{i,t}\le \beta \Bigg(\frac{{R_{M_{i,t}}}}{\alpha}\Bigg)^{2p},
\end{align}
where $\alpha=\eps/\eta$, ${R_{M_{i,t}}}= C + L_h E R_{i,t}$ and $R_{i,t}=\max_{j\in n_i}|\mu_{ij,t}|$. 
The model order of the multi-agent system is then
%
$M_t=\sum_{i=1}^N M_{i,t}$.
%
\end{thm}\vspace{-1mm}

Thm. \ref{thm:bound_memory_order} establishes an explicit non-asymptotic bound on the model complexity as a function of problem parameters as well as algorithm parameters $\epsilon$ and $\eta$. The proof of Thm. \ref{thm:bound_memory_order} is provided in  Appendix  \ref{app:proof_memoryorder}. Observe from \eqref{eq:agent_mo} that the model order depends on the algorithm parameters only through the ratio $\alpha = \eta/\epsilon$. In other words, by keeping $\alpha$ large, it is possible to ensure that the model order remains small. On the other hand, if we keep $\alpha$ small, the bound in \eqref{eq:agent_mo} would be large. It is remarked that the non-asymptotic characterization of the model complexity in the distributed non-parametric setting in Thm. \ref{thm:bound_memory_order} is the first of its kind. Existing results, such as in \cite{POLK}, only provide asymptotic guarantees. As an implication of Thm. \ref{thm:bound_memory_order}, observe that to ensure that the model order remains bounded by a given $M$, we must choose $\alpha = \ccalO(M^{-1/2p})$. Intuitively, if the model order is fixed a priori, the optimality gap cannot decay to zero in general, regardless of the number of iterations the algorithm is run. The resulting error floor introduced due to such a budget restriction will be characterized subsequently. 

Via Lemma \ref{thm:bound_gap}--\ref{lemma:inst_lagrang_diff} (see the supplementary material), we establish our central result, which is the mean convergence of Algorithm \ref{alg:soldd} in terms of sub-optimality and feasibility. 

 \vspace{-2mm}
\begin{thm}\label{thm:convergence}
Suppose Assumptions \ref{as:first}-\ref{as:fifth} hold  and  $(\bbf_t,\bbmu_t)$ be the primal-dual sequence of Algorithm \ref{alg:soldd} under constant step-size $\eta=T^{-1/2}$, and $\nu=\zeta T^{-1/2}+ \Lambda \alpha$, where $\zeta\geq \frac{1}{2}\bigg[R_{\ccalB}^2 + (1+\delta)\bigg(2 + 2 \Big(\frac{{4}VR_{\ccalB}(C X+\lambda R_{\ccalB})}{\xi}\Big)^{2}\bigg) + K\bigg]$ and $\Lambda\geq 4VR_{\ccalB}$ and ${K= 8V X^2 C^2 +4V\lambda^2 \cdot R_{\ccalB}^2+2EK_1+2EL_h^2X^2\cdot R_{\ccalB}^2}$. 

%
\begin{enumerate}[label=(\roman*)]
\item The average expected sub-optimality decays as 
\begin{align}\label{eq:func_order}
\frac{1}{T}\sum_{t=1}^T\mbE[S(\bbf_t)-S(\bbf^\star)]\leq\ccalO(T^{-1/2} + \alpha).
\end{align}
%
%
\item Moreover, the average of aggregate constraint is met, i.e.,   
 \begin{align}\label{eq:constr_order}
\!\!\!\!\!\!\!\!\!\!\frac{1}{T}\sum_{t=1}^T\mbE\Big[H_{ij}(f_{i,t},f_{j,t})-\gamma_{ij}\Big]\le 0, ~\forall ~(i,j)\in\ccalE.
\end{align}
%
\end{enumerate}
\end{thm}
%

The full proof of Thm. \ref{thm:convergence} is provided in Appendix \ref{app:proof_of_theorem}. Thm. \ref{thm:convergence} can be viewed an explicit non-asymptotic characterization of the optimality gap as a function of the iterations $T$ and the model complexity parameter $\alpha$. Recall from Thm. \ref{thm:bound_memory_order} that if $M$ is fixed a priori, it amounts to fixing $\alpha = \ccalO(M^{-1/2p})$. In other words, a fixed model order $M$ translates to a floor of $\ccalO(M^{-1/2p})$ on the optimality gap. 

In the extreme case when the model order is allowed to grow indefinitely, the optimality gap does decay as $\ccalO(T^{-1/2})$ while the constraint violation is zero. The rate provided here therefore improves upon the existing results in \cite{mahdavi2012trading,koppel2017proximity,madavan2019subgradient}, where the optimality gap is shown to decay as $\ccalO(T^{-1/4})$. Further, the rate provided here matches the best possible rate achievable for unconstrained general convex optimization problems. The specific approach responsible for allowing us to achieve zero constraint violation can be intuitively explained as follows: instead of the original problem, we actually solve a $\nu$-tightened problem with a smaller constraint set. As long as the original problem is strongly feasible and we set $\nu$ appropriately, such a tightening only leads to $\ccalO(T^{-1/2})$ suboptimality, so that the overall optimality gap only changes by a constant factor. 

It is instructive to express the results in Thm. \ref{thm:convergence} in terms of the iteration and model complexity parameters. Suppose that the goal is to achieve an desired optimality gap of $\varepsilon$. The following corollary, whose proof is provided in Appendix \ref{proof_iter_comp} of the supplementary material, characterizes the iteration and storage complexities of the proposed algorithm. 

{\begin{corollary}\label{thm:iter_comp}
For step-size $\eta=T^{-1/2}$ and dual regularizer $\nu=\zeta T^{-1/2} + \Lambda \alpha$, under Assumptions \ref{as:first}-\ref{as:fifth} hold, the optimality gap to be less than $\varepsilon$, i.e., $\frac{1}{T}\sum_{t=1}^T\mbE[S(\bbf_t)-S(\bbf^\star)]\leq \varepsilon $ requires  $\ccalO(1/\varepsilon^2)$ number of iterations and the model complexity bound of $\ccalO(1/\varepsilon^{2p})$ with the parameter $\alpha\leq \frac{\varepsilon}{2(4VR_{\ccalB} + Q \Lambda)}$, where  $Q=4VR_{\ccalB}(CX+\lambda R_{\ccalB})/\xi$.
\end{corollary}
As expected, the iteration complexity matches the state-of-the-art first order oracle complexity of SGD for the unconstrained convex setting. In addition, the model order complexity bound characterizes the corresponding storage requirements.}

\section{Numerical Results}\label{sec:simul}
In this section, we apply the proposed algorithm to solve a spatial temporal random field estimation problem and an another problem of inferring from oceanographic data.
\begin{figure*}[t]
\begin{subfigure}[b]{0.33\linewidth}
\centering
\includegraphics[width=\linewidth,height=0.64\linewidth]{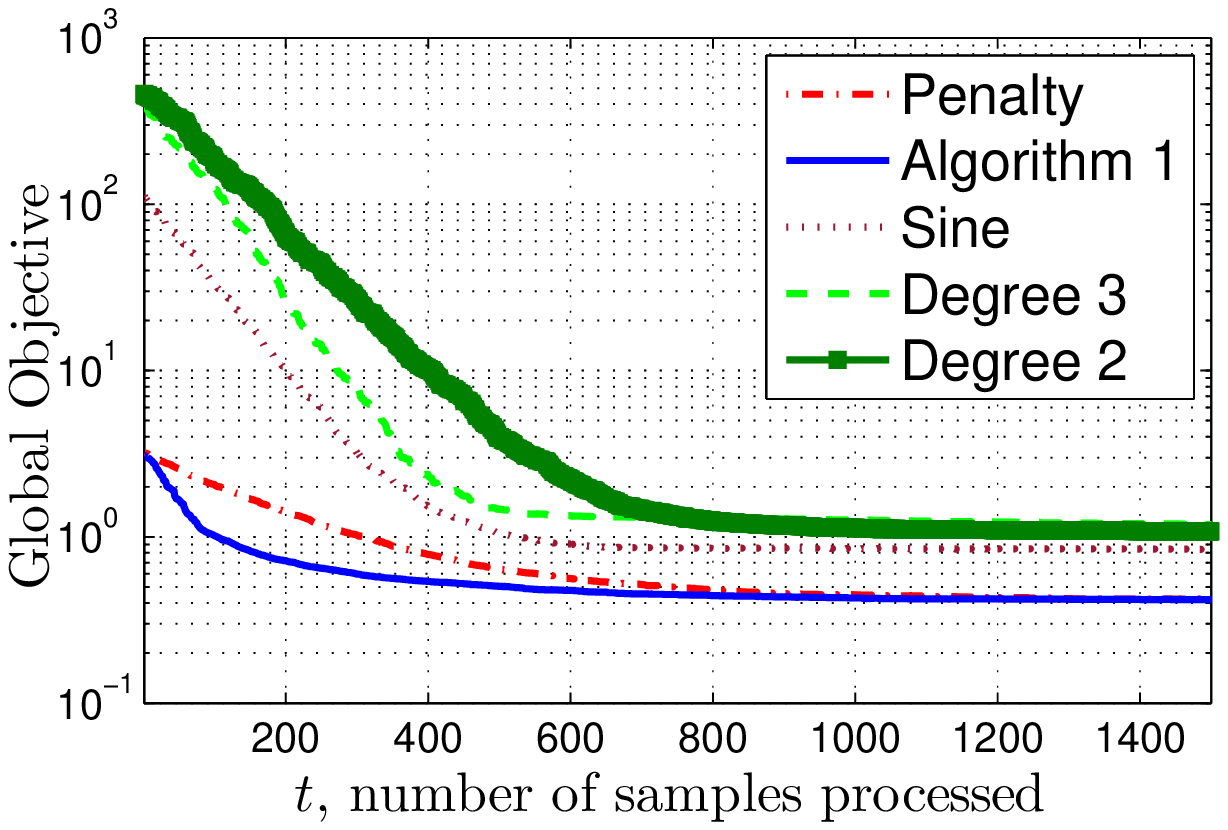}
\vspace*{-5mm}
\caption{\footnotesize{Global loss function vs. Samples}}
\label{fig:globalloss}
\end{subfigure}
\begin{subfigure}[b]{0.33\linewidth}
\centering
\includegraphics[width=\linewidth,height=0.6\linewidth]
{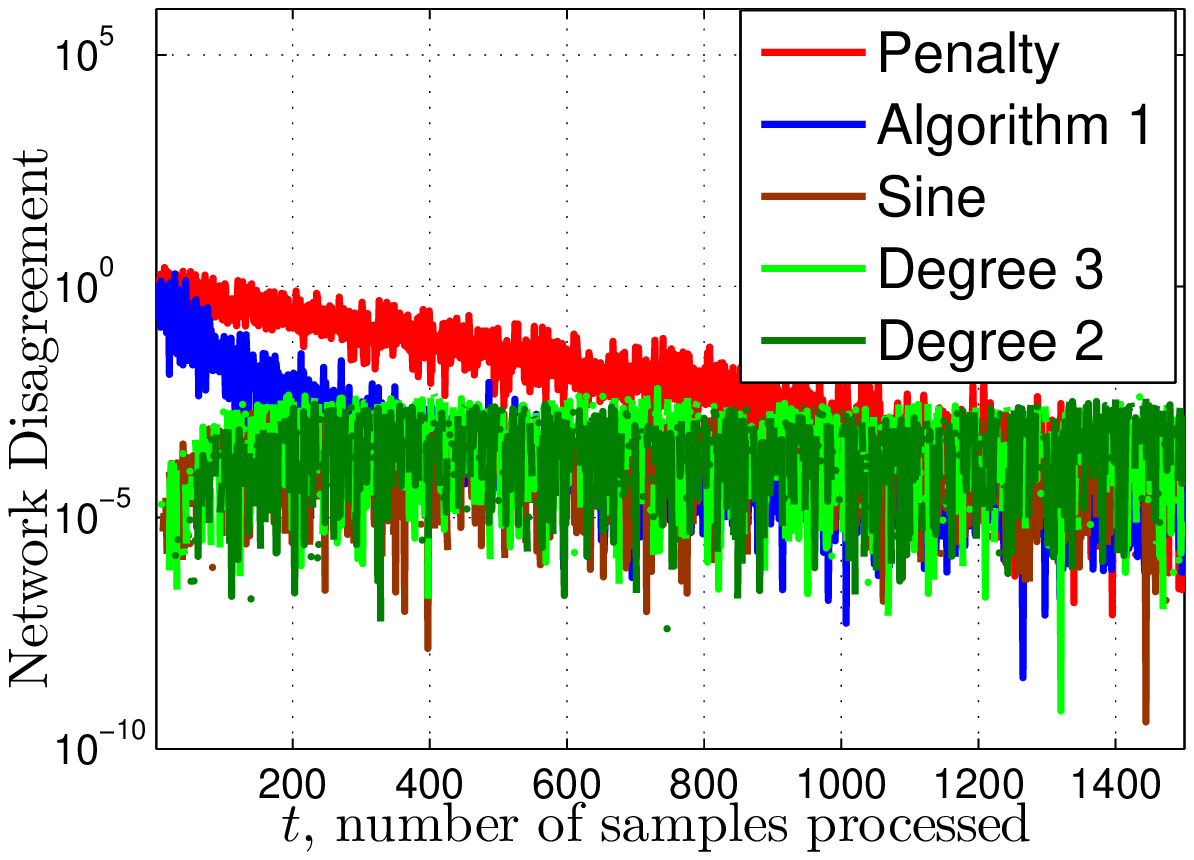}
\caption{\footnotesize{Network Disagreement vs. Samples}}
\label{fig:networkdisagreement}
\end{subfigure}
\begin{subfigure}[b]{0.33\linewidth}
\centering
\includegraphics[width=\linewidth,height=0.64\linewidth]
{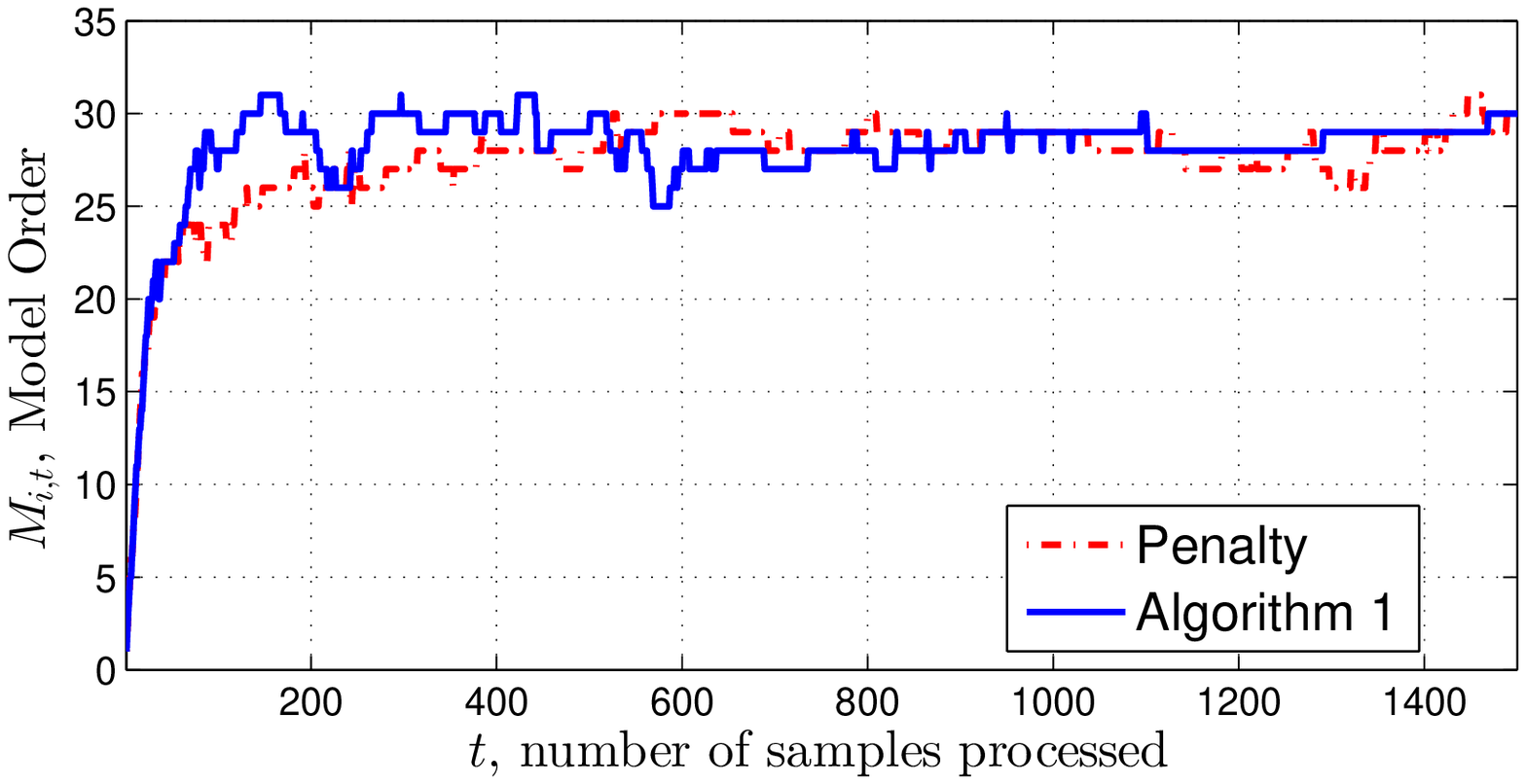}
\caption{\footnotesize{Model Order $\bbM_{i,t}$ vs. Samples}}
\label{fig:modelorder}
\end{subfigure}
\vspace*{-0.4cm}
\caption{\footnotesize{Convergence in terms of primal sub-optimality, constraint violation, and model complexity, for estimating a spatio-temporal correlated random field  with parsimony constant $P=8$, Gaussian kernel with bandwidth $0.05$, $\lambda=\delta=10^{-5}$, $\eta=0.01$ and penalty coefficient set at $0.08$. {Penalty method refers to an approximate constraint satisfaction approach to kernelized consensus optimization as in \cite{koppel2017decentralized}, whereas sine, degree 2, and degree 3 refer to linear statistical models over a fixed basis of sinusoids, or 2nd/3rd degree polynomials, to which primal-dual method is applied, as in \cite{koppel2017proximity}.} \vspace*{-0.4cm}}}
\label{fig:model_verify}
\end{figure*}
\vspace{-0.3cm}
\subsection{Spatio-temporal Random Field Estimation}\label{field_est}
The estimation of a spatio-temporal field using a set of sensors spread across a region with required level of accuracy is an central challenge in wireless sensor networks (WSNs) \cite{kozick2004source}. We model this problem by considering the problem of estimating a temporally varying spatial planar correlated Gaussian random field in a given region $\ccalG\subset\mbR^2$ space.  A spatial temporal random field is a random function of the spatial components  $u$ (for \emph{x}-axis) and $z$ (for \emph{y}-axis) across a region $\ccalG$ and time. Moreover, the random field is parameterized by its correlation matrix $\bbR_s$, which depends on the location of the sensors. Each element of $[\bbR_s]_{ij}$ is assumed to have a structure of the form $\Omega(l_i,l_j)=e^{-\|l_i-l_j\|}$, where $l_i$ and $l_j$ are the respective locations of sensor $i$ and $j$ in region $\ccalG$ \cite{kozick2004source}.  From this correlation, note that the nodes close to each other have high correlation whereas nodes located far away are less correlated, meaning that observations collected from the nearby nodes are  more relevant than observations from distant nodes.

We experimented with a sensor network with $V=40$ nodes spatially distributed in a $100\times 100$ meter square area. Each node $i$ collects the observation $y_{i,t}$ at time instant $t$. In the collected data, $y_{i,t}$ denotes the noisy version of the original field $s_{i,t}$ at node $i$ for time instant $t$.  The  observation model is given by $y_{i,t}=s_{i,t}+n_{i,t}$, where $n_{i,t}\sim\ccalN(0,0.5)$ is i.i.d. Each node seeks to sequentially minimize its local loss, i.e.,
\begin{align}\label{eq:loss_huber}
		\ell(\hat{s}_{i,t},y_{i,t})= \begin{cases}
		\frac{1}{2}(y_{i,t}-\hat{s}_{i,t})^2 &\text{for $|y_{i,t}-\hat{s}_{i,t}|\leq\phi $}\\
		\phi|y_{i,t}-\hat{s}_{i,t}|- \frac{1}{2}\phi^2&\text{otherwise}
	\end{cases}
\end{align}
 where $\hat{s}_{i,t}$ is the estimated value of actual field $s_{i,t}$ and $\phi$ is the tuning parameter. In the experiments, we choose $\delta=10^4$ to ensure that the local loss function is Lipschitz and the gradient is bounded by $2\delta$.  The instantaneous observation $\bbs_{t}$ across the network is given by $\bbs_{t}=\bbpi + \bbC^T\big(\mathbf{1}\sin(\omega t) + \bbv_t \big)$, where $\mathbf{1}$ is a vector of ones of length $V$, $\sin(\omega t)$ is a sinusoidal  with angular frequency $\omega=2$, $\bbpi=\{1/V,2/V,\dots,1\}$ is a fixed mean vector of length $V$, $\bbC$ is the Cholesky factorization of the correlation $\bbR_s$, and $\bbv_t\sim\ccalN(\mathbf{0},0.1\mathbf{I})$, where $I$ denotes $40\times40$ identity matrix. We select tolerance parameter to be $\gamma_{ij}=\Omega(l_i,l_j)$. 
We solve the problem \eqref{eq:main_prob} of minimizing the regularized loss function over $f_i$ where we learn the function $f_i(t)$, which is the function approximation to the actual field, $s_{i,t}$, i.e., we solve an online decentralized regression (curve fitting) problem.

 We run Algorithm \ref{alg:soldd} to generate local functions $f_i(t)$ to track ground-truth $y_{i,t}$. We select parameters: parsimony constant $P=8$ [Thm.  \ref{thm:convergence}], Gaussian kernel with bandwidth $\sigma=0.05$ such that we capture the variation of sinusoidal function with angular frequency of $2$, and primal/dual regularizers $\lambda=10^{-5}$ and $\delta=10^{-5}$. We run the algorithm for $1500$ iterations with step-size $\eta=0.01$. 
For comparison purposes, we consider two other comparable techniques: consensus with kernels via \textit{penalty method} \cite{koppel2017decentralized} and {the simplification of \eqref{eq:main_prob} to linear statistical models, i.e., $f_i(\bbx_i) = \bbw_i^T\phi(\bbx)$ for some $d$ fixed-dimensional parameter vectors $\{\bbw_i\}_{i=1}^N$,} which we call \textit{linear method} \cite{koppel2017proximity}. {For comparison with the penalty method, the penalty coefficient is $0.08$ which is tuned for the best performance, with all other parameters held fixed to Algorithm \ref{alg:soldd}.} For linear method, we consider three parametric models (which imply a different structural form for the feature map $\phi(\bbx)$): (a) Quadratic polynomial; (b) Cubic polynomial and (c) Sine polynomial (i.e., of the form $at+b\text{sin}(\omega t)$ where $a$ and $b$ are the model parameters and $\omega=1$ is the angular frequency). 
  \begin{figure*}[t]
  \begin{subfigure}[b]{0.33\linewidth}
  \centering
  \includegraphics[width=\linewidth,height=0.6\linewidth]{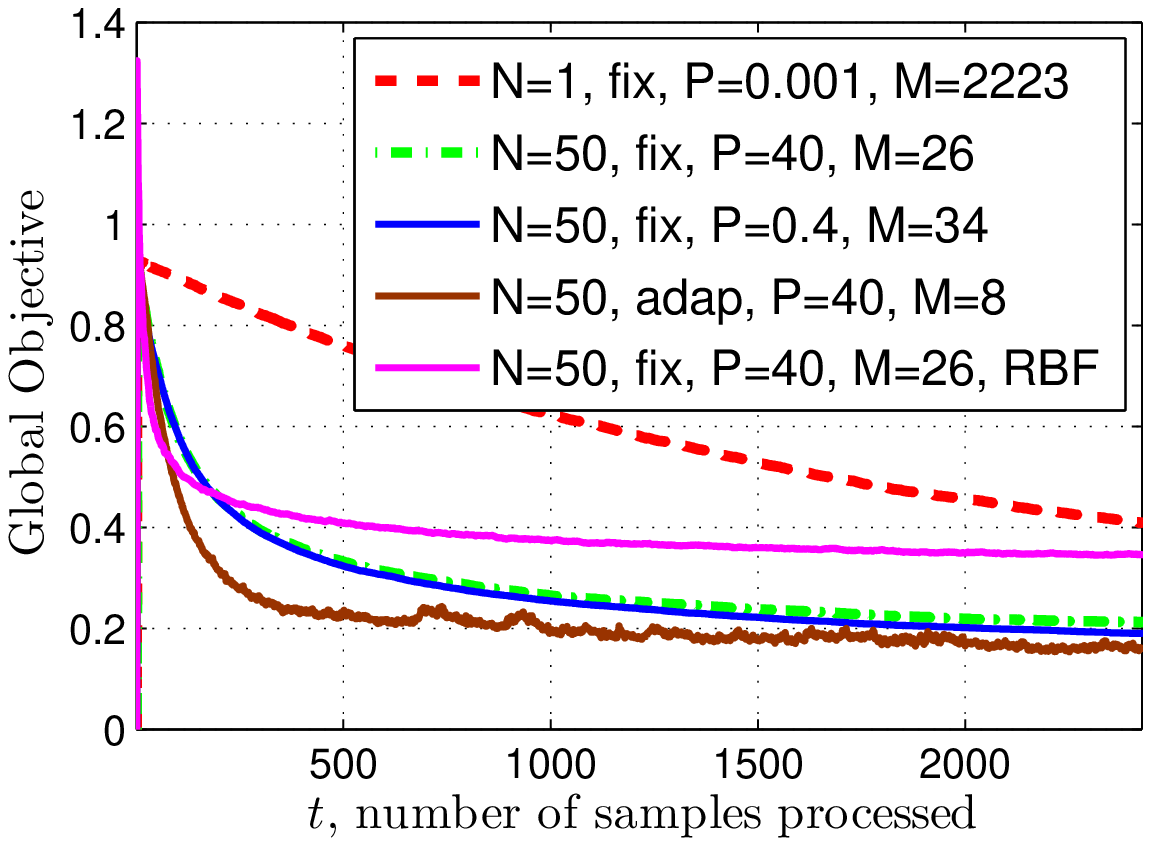}
  \caption{\footnotesize{Global loss function vs. Samples}}
  \label{fig:globalloss_temp_ocean}
  \end{subfigure}
  \begin{subfigure}[b]{0.33\linewidth}
  \centering
  \includegraphics[width=\linewidth,height=0.6\linewidth]
  {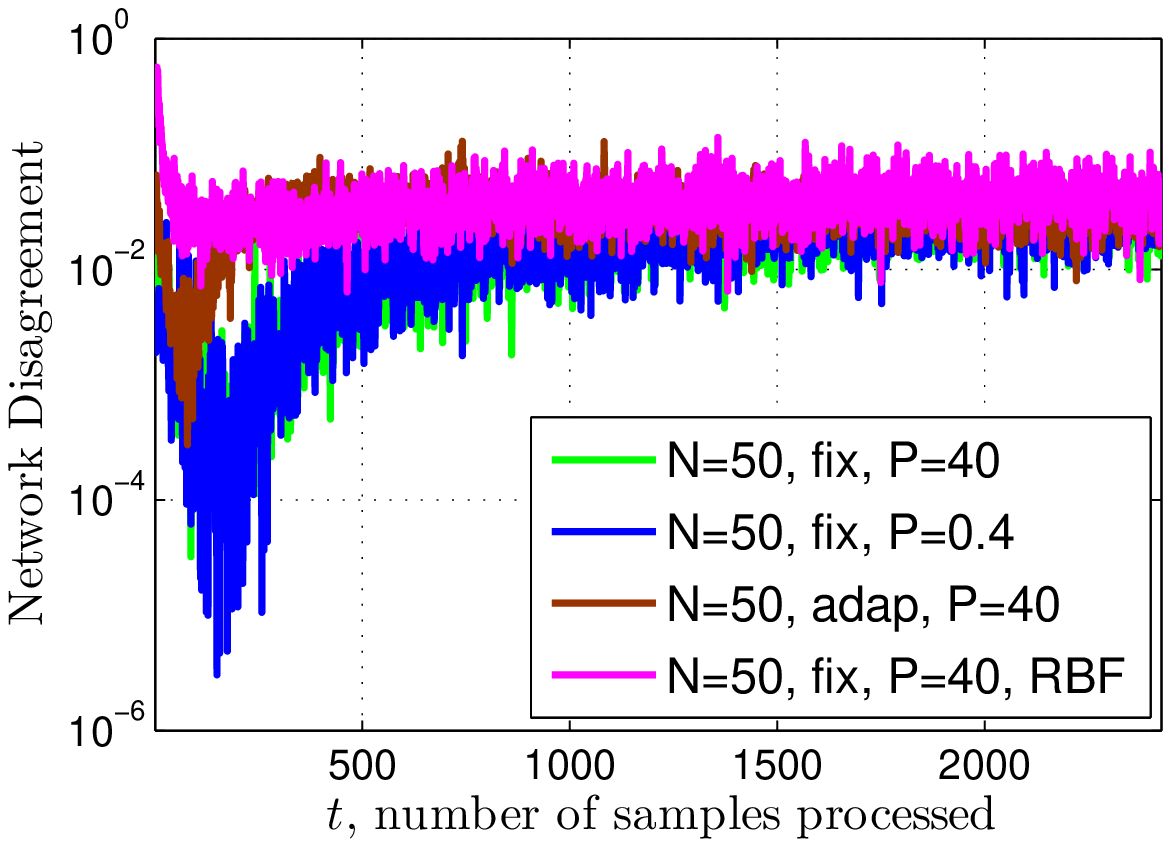}
  \caption{\footnotesize{Network Disagreement vs. Samples}}
  \label{fig:networkdisagreement_temp_ocean}
  \end{subfigure}
  \begin{subfigure}[b]{0.33\linewidth}
  \centering
  \includegraphics[width=\linewidth,height=0.6\linewidth]
  {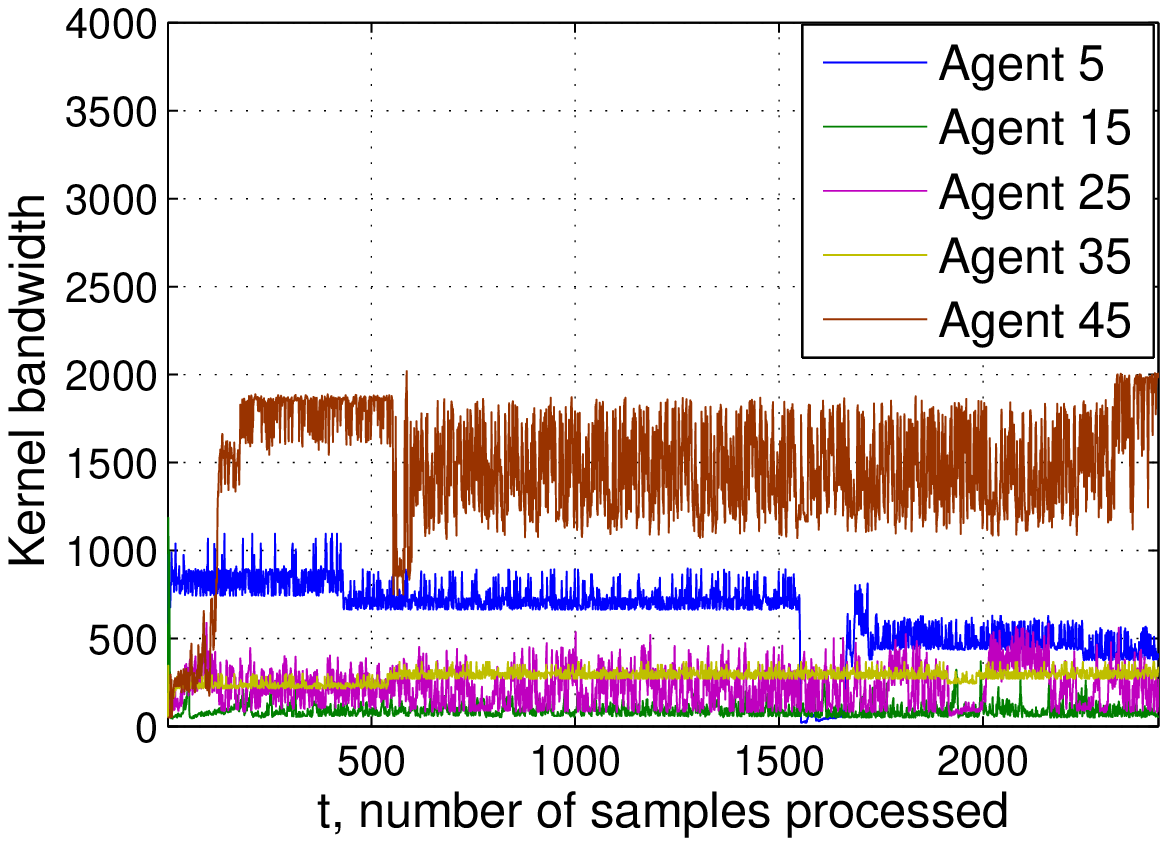}
  \caption{\footnotesize{Gaussian kernel bandwidth vs. Samples}}
  \label{fig:temp_agent_bw}
  \end{subfigure}
  \vspace*{-0.4cm}
  \caption{\footnotesize{Convergence in terms of primal sub-optimality and constraint violation, for temperature field of Gulf of Mexico with  $\lambda=\delta=10^{-5}$, and $\eta=0.01$. {The centralized solver pools all information at a single location and applies projected stochastic gradient method \cite{POLK}, whereas $P=40$ and $P=0.4$ are two different parameter selections for the parsimony constant in Algorithm \ref{alg:soldd}.}} Algorithm \ref{alg:soldd} is implemented with fixed  bandwidth $50$ and also variable kernel bandwidth for each agent, plotted in $(c)$ for five agents.  Algorithm \ref{alg:soldd} and its generalization attain the a favorable tradeoff of sub-optimality and feasibility. {$M$ denotes the model complexity of the algorithm, i.e., the number of dictionary elements (columns) in the kernel representation (matrix).}}
  \label{fig:model_verify_temp_ocean}
  \end{figure*}

 Fig. \ref{fig:model_verify} displays the results of this comparison. In particular, Fig. \ref{fig:globalloss} compares the global network wide averaged loss for three different methods, which shows that the linear method is unable to effectively track the target variable the model behavior. The linear sine polynomial model (denoted as ``Sine'') is closer to the target than the quadratic (denoted as ``Degree 2'') or cubic polynomial (denoted as ``Degree 3''). Algorithm \ref{alg:soldd} attains superior performance to the RKHS-based penalty method \cite{koppel2017decentralized}. Moreofer, Fig. \ref{fig:networkdisagreement} demonstrates that Algorithm \ref{alg:soldd} achieves tighter constraint satisfaction relative to penalty method, and is comparable to primal-dual schemes for linear models. Doing so allows nodes estimates tune their closeness to neighbors through proximity tolerances $\gamma_{ij}$. 
 
 Fig. \ref{fig:modelorder} plots the model order for primal-dual method and penalty method, which omits linear method plots because its a parametric method with fixed complexity equal to the parameter dimension $d$.
Early on, primal-dual method (being an exact method) has higher complexity than penalty method. In steady state, Algorithm \ref{alg:soldd} and penalty  method have comparable complexity. With a similar model complexity, we attain near-exact constraint satisfaction via primal-dual method as compared to penalty method. Overall, the model complexity of $30$ is orders of magnitude smaller than sample size  $1500$.


  \begin{figure*}[t]
  \begin{subfigure}[b]{0.33\linewidth}
  \centering
  \includegraphics[width=\linewidth,height=0.6\linewidth]{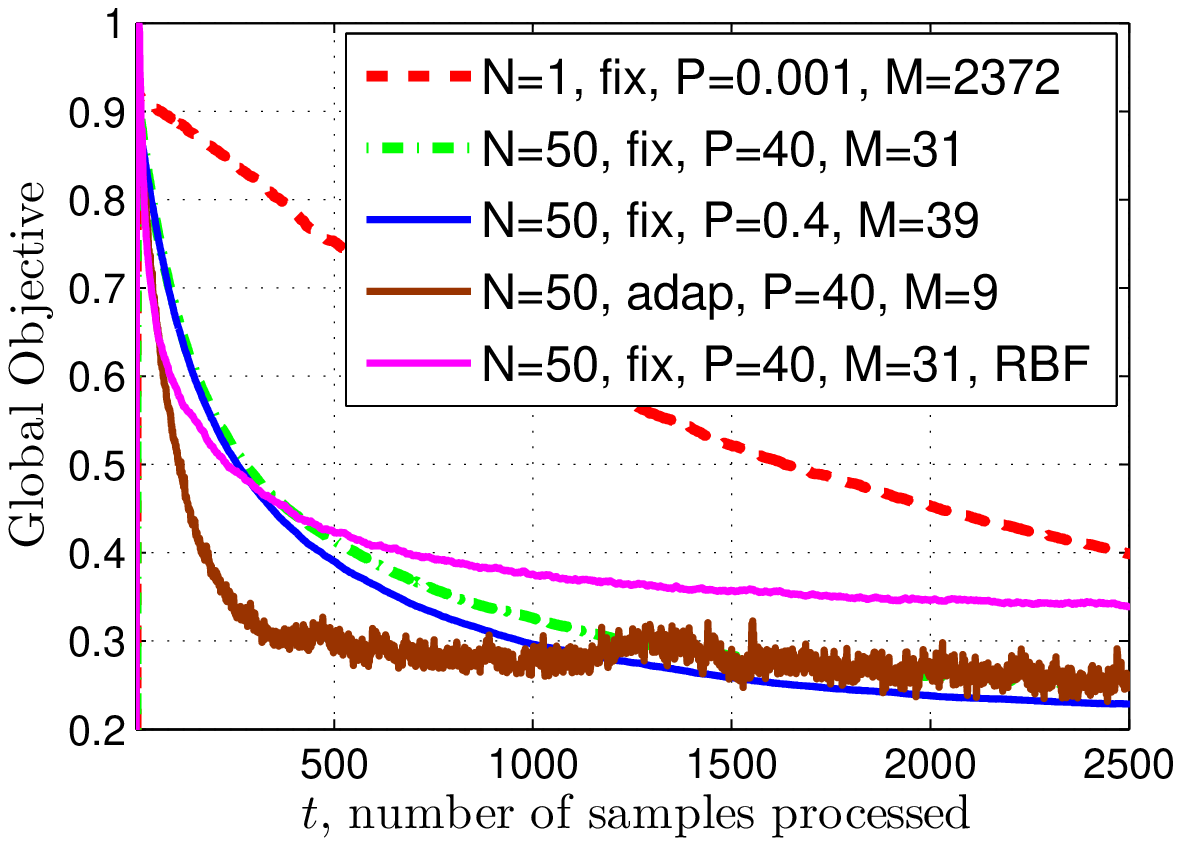}
  \caption{\footnotesize{Global loss function vs. Samples}}
  \label{fig:globalloss_salin_ocean}
  \end{subfigure}
  \begin{subfigure}[b]{0.33\linewidth}
  \centering
  \includegraphics[width=\linewidth,height=0.6\linewidth]
  {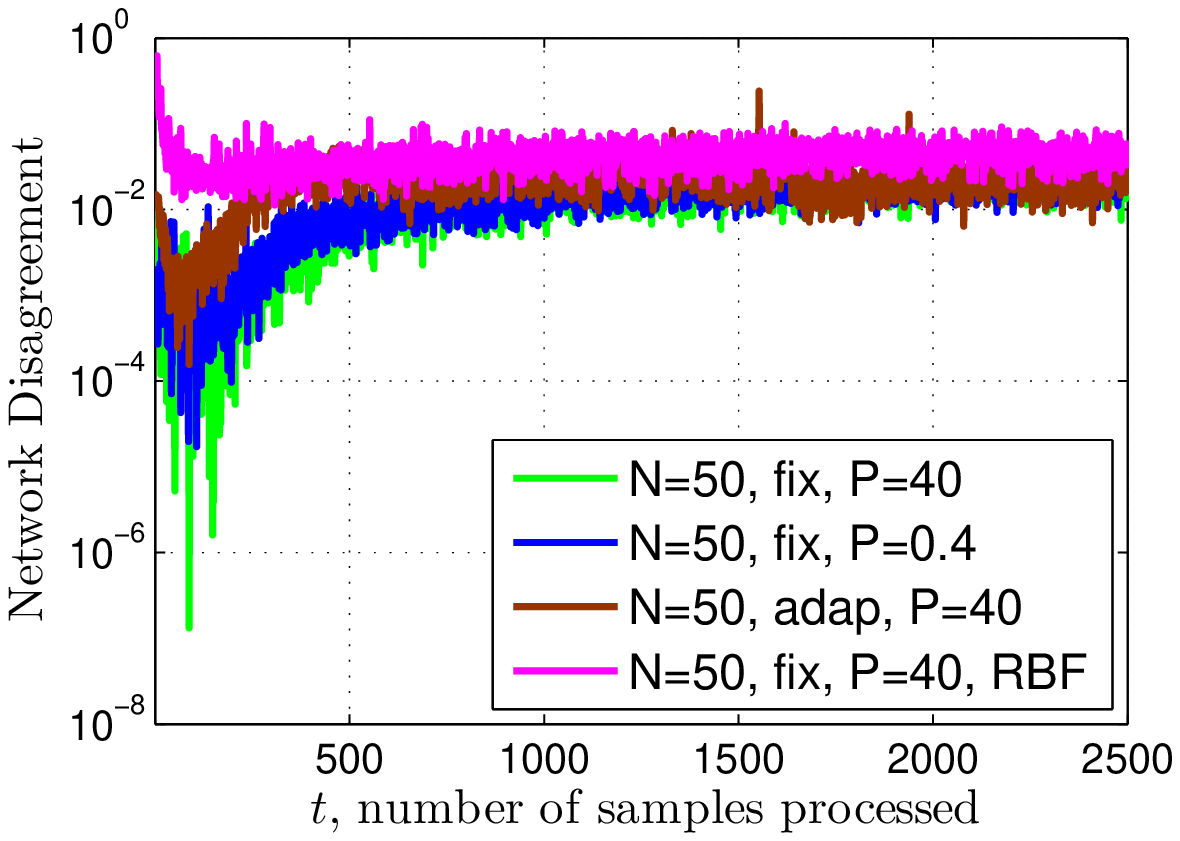}
  \caption{\footnotesize{Network Disagreement vs. Samples}}
  \label{fig:networkdisagreement_salin_ocean}
  \end{subfigure}
  \begin{subfigure}[b]{0.33\linewidth}
  \centering
  \includegraphics[width=\linewidth,height=0.6\linewidth]
  {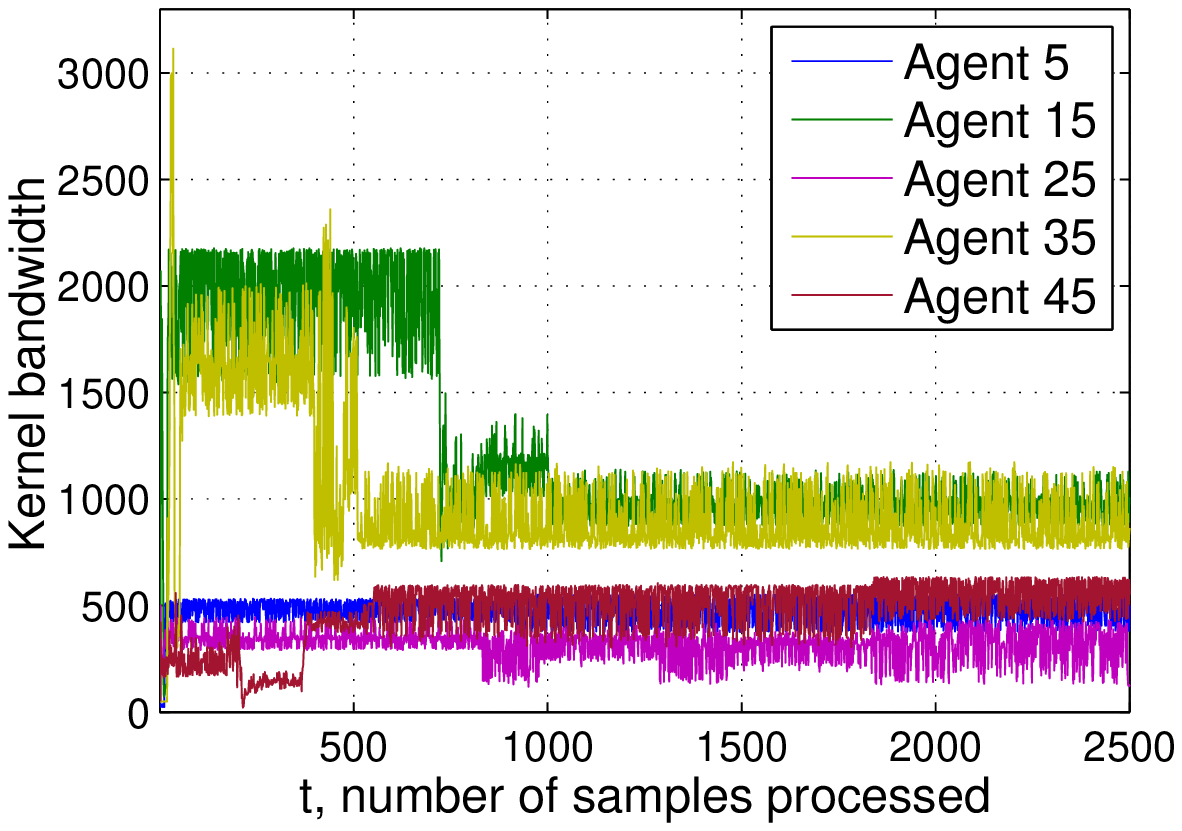}
  \caption{\footnotesize{Gaussian kernel bandwidth vs. Samples}}
  \label{fig:salinity_agent_bw}
  \end{subfigure}
  \vspace*{-0.4cm}
  \caption{\footnotesize{Convergence in terms of primal sub-optimality and constraint violation for salinity field of Gulf of Mexico with  $\lambda=\delta=10^{-5}$, and $\eta=0.01$. {The centralized solver pools all information at a single location and applying projected stochastic gradient method \cite{POLK}, whereas $P=40$ and $P=0.4$ are two different parameter selections for the parsimony constant in Algorithm \ref{alg:soldd}.} Algorithm \ref{alg:soldd} is run a Gaussian kernel with bandwidth $50$ and also adaptive bandwidth for each agent, plotted in $(c)$ for five agents. Algorithm \ref{alg:soldd} and its generalization most effectively trade off model fitness and constraint violation. {$M$ denotes the model order complexity of the algorithm.}}\vspace{-5mm}}
  \label{fig:model_verify_salin_ocean}
  \end{figure*}
 \vspace{-2mm} 
\subsection{Inferring Oceanographic Data}\label{sec:ocean}
Wireless sensor networks may also be used to monitor various environmental parameters, especially in oceanic settings. To this end, we associate each node in the network to an oceanic buoy tasked with estimating salinity and temperature when deployed at {standard} depths. Decentralization is advantageous here due to the fact that server stations are impractical at sea, and centralization may exceed the cost of computation per node \cite{kozick2004source}. Thus, we run Algorithm \ref{alg:soldd} on the World Oceanic Database \cite{boyer}, obtained from multiple underwater sensors in the Gulf of Mexico. In this Regional Climatology data set, temperature, salinity, oxygen, phosphate, silicate, and nitrate are recorded at various depth levels.  

We restrict focus to temperature and salinity parameters at different locations with varying depths, during the winter time-period. The readings of the climatological fields are obtained for a particular latitude and longitude at standard depths starting from $0$ meters to $5000$ meters. The latitude and longitude specifies the node (sensor) location. Similar readings are obtained for various locations spanning the water body. The experiment is carried out considering $50$ nodes, where edges are determined by measuring the distance between two nodes, and drawing an edge to a particular node if its distance is less than $1000$ kilometers away. The proximity parameter $\gamma_{ij}$ is obtained by evaluating $\exp(-\text{dist}(i,j)/1000)$, where $\text{dist}(i,j)$ denotes distance nodes in kilometers. 

 We use Algorithm \ref{alg:soldd} to estimate the value of climatological field $y_i$ at a depth $d_i$ such as salinity or temperature at each node $i$. We solve problem \eqref{eq:main_prob} by minimizing the regularized quadratic loss between estimated climatological field and observed climatological field $y_i$ over function $f_i$.  A key benefit of doing so is the ability to interpolate missing values, which arise due to, e.g., limited battery or bandwidth.
 
 {
 {\noindent \bf Bandwidth Adaptation}
We consider an extended implementation of Algorithm \ref{alg:soldd} for the purpose of experimentation, where kernel hyper-parameters may be selected adaptively via online maximum likelihood, rather than fixed a priori \cite{singh2011information}. The motivation is that in this oceanic setting, spatial fields exhibit scale heterogeneity that can be exploited for boosting accuracy.
%
%
In the decentralized online setting, we run a bandwidth at step $(t+1)$ for the Gaussian kernel of agent $i$ as:\vspace{-1mm}
%
\begin{align}\label{eq:var_bw}
\!\sigma_{i,t+1}\!\!=\!\!\sqrt{\!\frac{1}{M_{i,t}\!+\!1}\!\!\!\sum_{l=1}^{M_{i,t}+1}\!\frac{\sum_{k=1,k\neq l}^{M_{i,t}+1} \exp\!\Big(\!\!-\!\frac{(\bbd_l-\bbd_k)^2}{2\sigma_{i,t}^2}\!\Big)(\bbd_l\!-\!\bbd_k)^2}{\sum_{k=1,k\neq l}^{M_{i,t}+1} \exp\!\Big(\!-\!\frac{(\bbd_l-\bbd_k)^2}{2\sigma_{i,t}^2}\Big)}}
\end{align} 
 where $\bbd_l$ and $\bbd_k$ are the dictionary elements of $\tbD_{i,t+1}$ [cf \eqref{eq:param_tilde}]. The utility of this adaptive bandwidth selection is investigated next. For better interpretability, we denote the centralized method with fixed kernel bandwidth as ``N=1, fix'' and distributed method with fixed and adaptive bandwidth as ``N=50, fix'' and ``N=50, adap''.}
 
 \subsubsection{Temperature}\label{temperature}
 Here we use Algorithm \ref{alg:soldd} for predicting the statistical mean of the temperature field  of different nodes at varying depths. The real data obtained from the World Oceanic database has statistical mean of the temperature field. 
 %
 We run Algorithm \ref{alg:soldd} for $T=2430$ iterations with constant step-size $\eta=0.01$ and regularizers $\lambda=10^{-5}$, $\delta=10^{-5}$ with a Gaussian kernel with fixed bandwidth parameter $\sigma=50$. The adaptive scheme employs \eqref{eq:var_bw} with the same bandwidth initialization. The parsimony constant is fixed at two values, $P\in\{0.4,40\}$. {The adaptive bandwidth case is studied only for $P=40$. The parsimony constant is set to $P=0.001$ for \cite{POLK} to ensure comparably sized models across cases.}
 
 Fig. \ref{fig:model_verify_temp_ocean} shows the numerical experiment results for the ocean data. Fig. \ref{fig:globalloss_temp_ocean} demonstrates that the prediction error for test cases reduces with increasing samples, and illustrates that centralization is inappropriate here: local spatial variability of the field causes \cite{POLK} to be outperformed by Algorithm \ref{alg:soldd} under both fixed and variable bandwidths. {The variable bandwidth along with appropriately chosen parsimony constant and step size performs best in terms of model fitness.} A similar trend may be gleaned from the plot of constraint violation over time in Fig. \ref{fig:networkdisagreement_temp_ocean}. {Interestingly, the adaptive bandwidth scheme obtains more accurate model with comparable constraint violation, both of which are superior to centralized approaches, thus substantiating the experimental merits of \eqref{eq:var_bw}. The evolution of a few random agents' bandwidths is visualized in {Fig. \ref{fig:temp_agent_bw}} -- since \eqref{eq:var_bw} is a stochastic fixed point iteration, each agent's bandwidth converges to a neighborhood. }

 \subsubsection{Salinity}\label{salinity}
 
Next we consider Algorithm \ref{alg:soldd} for predicting the mean salinity from various oceanic locations and depths. We set the primal and dual regularizer $\lambda=\delta=10^{-5}$, and run it for $T=2500$ iterations with constant step-size of $0.01$. We select two values of the parsimony constant $P\in\{0.4,40\}$. For the centralized case, we fix $P=0.001$ so that its complexity is comparable to the distributed approach to ensure a fair comparison. The bandwidth of the Gaussian kernel is fixed at $50$ and also considered adaptive for the simulation with initial bandwidth value set at $50$. 
  
We display these results in Fig. \ref{fig:model_verify_salin_ocean}. Fig. \ref{fig:globalloss_salin_ocean} demonstrates that learning a single function to fit all data is unable to filter out correlation effects and hence gives poor model fitness, as compared to fitting multiple $f_i$'s to different nodes considering proximity constraints. Moreover, the average model order for a single node for the distributed case is $39$ for $P=0.4$, thus giving an aggregate complexity of $1950$ for $50$ nodes. This is less than the centralized model order of $2372$ for a single function. {Thus the distributed approach yields improved accuracy with reduced complexity.}
 We note that increasing the parsimony constant results in worse model fit but saves complexity, yielding a tunable tradeoff between fitness and complexity. Similar to the temperature data, in Fig. \ref{fig:globalloss_salin_ocean}  also we observe  improvement in performance for adaptive bandwidth case with smaller average model order. In Fig. \ref{fig:networkdisagreement_salin_ocean} we may observe that Algorithm \ref{alg:soldd} incurs attenuating constraint violation   for both fixed and adaptive bandwidth case. Overall, complexity settles to around $39$, which is orders of magnitude smaller than the $2500$ sample size.


{\subsubsection{Radial basis function (RBF) network}\label{sec:rbf}

To evaluate the merit of the proposed decentralized projection scheme, we considered a variant of our approach which employs a fixed feature representation, i.e., linear regression over a nonlinear RBF basis whose points are selected according to a uniform grid across the feature space. Doing so may be interpreted as a compatible comparison to \cite{koppel2017proximity}. We call this scheme RBF, as it is a multi-agent proximity-constrained problem for learning the coefficients of a one-layer RBF network. We implement this scheme on the real ocean data set considered in the paper with the same parameters as considered for HALK.

The fixed model order for a single agent for the RBF case is fixed to the final settled model order obtained in case of our paper (for parsimony constant $P=40$), i.e., 26 and 31 for temperature and salinity field estimation. The fixed dictionary points for the RBF case are randomly selected from a uniform distribution.
  It can be observed from both Fig. \ref{fig:globalloss_temp_ocean} and Fig. \ref{fig:globalloss_salin_ocean} that for salinity and temperature field estimation, the RBF based algorithm  performs inferior to our algorithm in the long run. IEarly on, RBF performs better since it has fewer parameters to learn due to its fixed basis, but over time the quality of its fixed basis is exceeded by the dynamically chosen one of HALK. The constraint violation for both the field estimation parameters are presented in Fig. \ref{fig:networkdisagreement_temp_ocean} and Fig. \ref{fig:networkdisagreement_salin_ocean}.   
  The important observation here is that our proposed algorithm discerns the number of dictionary points (model order) dynamically and it evolves with the system. However, in case of RBF based algorithm, we require to apriori fix the size of the RBF network for the algorithm to initialize and iterate, which in many cases is difficult to fix in advance.}

\vspace{-2mm}
\section{Conclusion} \label{sec:conclusion}
We proposed learning in heterogeneous networks via constrained functional stochastic programming. We modeled heterogeneity using proximity constraints, which allowed agents to make decisions which are close but not necessarily equal. Moreover, motivated by their universal function approximation properties, we restricted focus to the case where agents decisions are defined by functions in RKHS.
We formulated the augmented Lagrangian, and proposed a decentralized stochastic saddle point method to solve it.
Since decision variables were functions belonging to RKHS, not vectors, we required generalizing the Representer Theorem to this setting, and further projecting the primal iterates onto subspaces greedily constructed from subsets of past observations. 

The algorithm, Heterogeneous Adaptive Learning with Kernels (HALK), converges in terms of primal sub-optimality and satisfies the constraints {on average}. We further established a controllable trade-off between convergence and complexity. 
We validated HALK for estimating a spatial temporal Gaussian random field in a heterogeneous sensor network, and employed it to predict oceanic temperature and salinity from depth. We also considered a generalization where the kernel bandwidth of each agent's function is allowed to vary, and observed better experimental performance as compared to the fixed bandwidth.
 In future work, we hope to relax communications requirements and allow asynchronous updates \cite{bedi2019asynchronous,bedi2017beyond}, permit the learning rates to be distinct among agents, and obtain tighter dependence of the convergence results on the network data.

%
\vspace{-2mm}
\renewcommand{\thesectiondis}[2]{\Alph{section}:}
\appendices

\section{Proof of Thm. \ref{thm:bound_memory_order}}\label{app:proof_memoryorder}
The proof is motivated from the derivation presented in \cite[Thm. 4]{POLK}  and is presented here for the proposed algorithm. Consider the function iterates $f_{i,t}$ and $f_{i,t+1}$ of agent $i$ generated from Algorithm \ref{alg:soldd} at $t$-th and $(t+1)$-th instant. The function iterates $f_{i,t}$ and $f_{i,t+1}$ are parametrized  by dictionary $\bbD_{i,t}$ and $\bbD_{i,t+1}$ and weights $\bbw_{i,t}$ and $\bbw_{i,t+1}$, respectively. The dictionary size corresponding to $f_{i,t}$ and $f_{i,t+1}$ in dictionary $\bbD_{i,t}$ and $\bbD_{i,t+1}$ are denoted by $M_{i,t}$ and $M_{i,t+1}$, respectively. The kernel dictionary $\bbD_{i,t+1}$ is formed from $\tbD_{i,t+1}= [\bbD_{i,t},\;\;\bbx_{i,t}]$ by selecting a subset of  $M_{i,t+1}$ columns from $\tilde{M}_{i,t+1}=M_{i,t}+1$ number of columns of $\tbD_{i,t+1}$  that best approximate $\tilde{f}_{i,t+1}$ in terms of Hilbert norm error, i.e., $\|f_{i,t+1} - \tilde{f}_{i,t+1} \|_{\ccalH} \leq \eps $, where $\eps$ is the error tolerance.
Suppose the model order of function $f_{i,t+1}$ is less than equal to that of $f_{i,t}$, i.e., $M_{i,t+1}\le M_{i,t}$, which holds when the stopping criteria of KOMP is violated for dictionary $\tbD_{i,t+1}$:
\begin{align}\label{eq:dist_eq2}
\min_{j=1,\dots,M_{i,t}+1} \gamma_j\le \eps,
\end{align}
 where $\gamma_j$ is the minimal approximation error with dictionary element $\bbd_{i,j}$ removed from dictionary $\tbD_{i,t+1}$ defined as
\begin{align}\label{eq:dist_eq3}
\gamma_j=\min_{\bbw\in \reals^{\tilde{M}_{i,t+1}-1}} \|\tilde{f}_{i,t+1}(\cdot)-\sum_{k\in \ccalI \setminus \{j\}} w_k\kappa(\bbd_{i,k}, \cdot)\|_\ccalH,
\end{align} 
 where $\ccalI=\{1,\dots,M_{i,t}+1\}$. 
 
Observe that \eqref{eq:dist_eq2} lower bounds the approximation error $\gamma_{M_{i,t}+1}$ of removing the most recently added feature vector $\bbx_{i,t}$. Thus if $\gamma_{M_{i,t}+1}\le \eps$, then \eqref{eq:dist_eq2} is satisfied and the relation $M_{i,t+1}\le M_{i,t}$ holds, implying the model order does not grow. Hence it is adequate to  consider $\gamma_{M_{i,t}+1}$.

Using the definition of $\tilde{f}_{i,t+1}$ and denoting $\ccalI'\coloneqq \ccalI \setminus \{M_{i,t}+1\}$, we write $\gamma_{M_{i,t}+1}$ as  
\begin{align}\label{eq:dist_eq4}
&\gamma_{M_{i,t}+1}\\
&=\min_{\bbu\in \reals^{M_{i,t}}} \|{f}_{i,t} -\eta \nabla_{f_i}\hat{\ccalL}_{t}(\bbf_t,\bbmu_t)-\sum_{k\in \ccalI'}  u_k\kappa(\bbd_{i,k}, \cdot)\|_\ccalH. \nonumber
\end{align}
The minimizer of \eqref{eq:dist_eq4} is obtained for $\bbu^*$ is obtained via a least-squares computation, and takes the form:
%
%
\begin{align}\label{eq:dist_eq6}
\bbu^*&=(1-\eta\lambda)\bbw_{i,t}-\eta\Big[\ell_i'(f_{i,t}(\bbx_{i,t}),y_{i,t}) \\
&+ \sum_{j\in n_i}\mu_{ij,t}h'_{ij}(f_{i,t}(\bbx_{i,t}),f_{j,t}(\bbx_{i,t}))\Big]\bbK_{\bbD_{i,t},\bbD_{i,t}}^{-1}\boldsymbol{\kappa}_{\bbD_{i,t}}(\bbx_{i,t}).\nonumber
\end{align}
Further, we define $\ell_i'(f_{i,t})$  as
\begin{align}\label{eq:l_prime}
\ell_i'(f_{i,t}):=&\ell_i'(f_{i,t}(\bbx_{i,t}),y_{i,t}) \nonumber\\
&+ \sum_{j\in n_i}\mu_{ij,t}h'_{ij}(f_{i,t}(\bbx_{i,t}),f_{j,t}(\bbx_{i,t}))\big],
\end{align}
Substituting $\bbu^*$ from \eqref{eq:dist_eq6} into \eqref{eq:dist_eq4} and applying Cauchy-Schwartz with \eqref{eq:l_prime}, we obtain $\gamma_{M_{i,t}+1}$ as 
\begin{align}\label{eq:dist_eq8}
\gamma_{M_{i,t}+1} 
&\le |{\eta\ell_i'(f_{i,t})}| \\
&\times \Big\lVert \kappa(\bbx_{i,t},\cdot)-\Big[\bbK_{\bbD_{i,t},\bbD_{i,t}}^{-1}\boldsymbol{\kappa}_{\bbD_{i,t}}(\bbx_{i,t})\Big]^T\boldsymbol{\kappa}_{\bbD_{i,t}}(\cdot)\Big\rVert_\ccalH.\nonumber
\end{align}
It can be observed from the right hand side of \eqref{eq:dist_eq8} that the Hilbert-norm term can be replaced by using the definition of subspace distance from \eqref{eq:dist1}. Thus, we get
\begin{align}\label{eq:dist_eq9}
\gamma_{M_{i,t}+1} \le |{\eta\ell_i'(f_{i,t})}| ~\text{dist}(\kappa(\bbx_{i,t}, \cdot), \ccalH_{\bbD_{i,t}}).
\end{align}
Now for $\gamma_{M_{i,t}+1}\le \eps$, the right hand side of \eqref{eq:dist_eq9} should also be upper bounded by $\eps$ and thus can be written as
\begin{align}\label{eq:dist_eq10}
\text{dist}(\kappa(\bbx_{i,t}, \cdot), \ccalH_{\bbD_{i,t}}) \le \frac{\eps}{\eta|{\ell_i'(f_{i,t})}|},
\end{align}
where we have divided both the sides by $|{\ell_i'(f_{i,t})}|$. Note that if \eqref{eq:dist_eq10} holds, then $\gamma_{M_{i,t}+1}\le \eps$ and since $\gamma_{M_{i,t}+1}\geq \min_j \gamma_j$, we may conclude that  \eqref{eq:dist_eq2} is satisfied. Implying the model order at the subsequent steps does not grow, i.e., $M_{i,t+1}\le M_{i,t}$.

Now, let's take the contrapositive  of the expression in \eqref{eq:dist_eq10} to observe that growth in the model order ($M_{i,t+1}= M_{i,t}+1$) implies that the condition
\begin{align}\label{eq:dist_eq11}
\text{dist}(\kappa(\bbx_{i,t}, \cdot), \ccalH_{\bbD_{i,t}}) > \frac{\eps}{\eta|{\ell_i'(f_{i,t})}|},
\end{align}
holds. Therefore, every time a new point is added to the model, the corresponding kernel function, i.e., $\kappa(\bbx_{i,t}, \cdot)$ is at least $\frac{\eps}{\eta|{\ell_i'(f_{i,t})}|}$ distance far away from  every other kernel
function in the current model defined by dictionary $\bbD_{i,t}$.

 Now, to have a bound on the right-hand side term of \eqref{eq:dist_eq11}, we bound the denominator of the right-hand side of  \eqref{eq:dist_eq11}. Thus, we upper bound $|{\ell_i'(f_{i,t})}|$ as
\begin{align}\label{eq:dist_eq12_temp}
&|{\ell_i'(f_{i,t})}|\nonumber\\
&=|{\big[\ell_i'(f_{i,t}(\bbx_{i,t}),y_{i,t}) + \sum_{j\in n_i}\mu_{ij,t}h'_{ij}(f_{i,t}(\bbx_{i,t}),f_{j,t}(\bbx_{i,t}))\big]}|\nonumber\\
&\le |\ell_i'(f_{i,t}(\bbx_{i,t}),y_{i,t})| + | \sum_{j\in n_i}\mu_{ij,t}h'_{ij}(f_{i,t}(\bbx_{i,t}),f_{j,t}(\bbx_{i,t}))|\nonumber\\
& \le C + L_h E R_{i,t},
\end{align}
where in the first equality we have used the definition of $\ell_i'(f_{i,t})$ from \eqref{eq:l_prime} and  the second inequality  is obtained by using triangle inequality. To obtain the last inequality in \eqref{eq:dist_eq12_temp}, we use Assumption \ref{as:second} and \ref{as:third}, with  $R_{i,t}=\max_{j\in n_i}|\mu_{ij,t}|$ and upper-bounded $|n_i|$ by the total number of edges $E$. Subsequently, we denote the right hand side of  \eqref{eq:dist_eq12_temp} as ${R_{M_{i,t}}}\coloneqq C + L_h E R_{i,t}$. Substituting into \eqref{eq:dist_eq11} yields
%
%
\begin{align}\label{eq:dist_eq13}
\text{dist}(\kappa(\bbx_{i,t}, \cdot), \ccalH_{\bbD_{i,t}}) > \frac{\eps }{\eta {R_{M_{i,t}}}}.
\end{align}
Hence, the stopping criterion for the newest point is violated whenever it satisfies the condition, $\|\phi(\bbx_{i,t})-\phi(\bbd_{i,k})\|_2\le \frac{\eps }{\eta {R_{M_{i,t}}}}$ for $k\in\{1,\dots,M_{i,t}\}$ meaning $\phi(\bbx_{i,t})$ can be well approximated by $\phi(\bbd_{i,k})$ with already existing point $\bbd_{i,k}$ in the dictionary. Now for the finite model order proof, we proceed in a manner similar to the proof of \cite[Thm. 3.1]{1315946}. Since feature space $\ccalX$ is compact and $\phi(\bbx)=\kappa(\bbx,\cdot)$ is continuous (as $\kappa$ is continuous), we can deduce that $\phi(\ccalX)$ is compact. Therefore, the covering number (number of balls with radius $\varrho=\frac{\eps }{\eta {R_{M_{i,t}}}}$ to cover the set $\phi(\ccalX)$) of set $\phi(\ccalX)$ is finite \cite{zhou2002covering}.  The covering number of a set is finite if and only if its packing number (maximum number of points in $\phi(\ccalX)$ separated by distance larger than $\varrho$) is finite. This means that the number of points in $\phi(\ccalX)$ separated by distance $\varrho$ is finite. Note that KOMP retains points which satisfy $\|\phi(\bbd_{i,j})-\phi(\bbd_{i,k})\|_2 >\varrho$, i.e., the dictionary points are $\varrho$ separated. Thus, when the packing number of $\phi(\ccalX)$ with scale $\varrho$ is finite, the number of dictionary points is also finite.

From \cite[Proposition 2.2]{1315946}, we know that for a Lipschitz continuous Mercer
kernel $\kappa$ on a compact set $\ccalX \subset \reals^p$, there exists a constant $\beta$ depending upon the $\ccalX$ and the kernel function such that for any training set $\{\bbx_{i,t}\}_{t=1}^\infty$ and any $\vartheta>0$,  the number of elements in the dictionary satisfies 
\begin{align}\label{eq:dist_eq14}
M_{i,t}\le \beta \Big(\frac{1}{\vartheta}\Big)^{2p}.
\end{align}
From \eqref{eq:dist_eq13}, observe $\vartheta=\frac{\eps }{\eta {R_{M_{i,t}}}}$. With \eqref{eq:dist_eq14} we may write
%
$M_{i,t}\le \beta ({\eta {R_{M_{i,t}}}}/{\eps})^{2p}$,
%
which is the required result stated in  \eqref{eq:agent_mo} of Thm. \ref{thm:bound_memory_order}. To obtain the model order  $M_t$ of the multi-agent system, we sum the model order of individual nodes across network, i.e.,
%
$M_t=\sum_{i=1}^N M_{i,t}$
%
as stated in Thm. \ref{thm:bound_memory_order}.\hfill $\blacksquare$

\section{Proof of Thm. \ref{thm:convergence}}
\label{app:proof_of_theorem}
%
Before discussing the proof, we introduce the following compact notations to make the analysis clear and compact. 
 We further use the following short-hand notations to denote the expressions involving $h_{ij}(\cdot,\cdot)$ as
\begin{align}\label{notation}
g_{ij}(f_t(\bbx_t)):=& h_{ij}(f_{i,t}(\bbx_{i,t}),f_{j,t}(\bbx_{i,t}))-\gamma_{ij}
\end{align} 
Moreover, the $E$-fold stacking of the constraints across all the edges is denoted as $\bbG^t(\bbf)\coloneqq \text{vec}[g_{ij}(f(\bbx_t))]$ and $\bbG(\bbf)\coloneqq \mbE_{\bbx_t}[\bbG^t(\bbf)]$. The intermediate results required for the proof  are stated in Lemma \ref{thm:bound_gap}-\ref{lemma:inst_lagrang_diff} (detailed in the  supplementary material).  Consider the statement of Lemma \ref{lemma:inst_lagrang_diff} (c.f. \eqref{eq:inst_lagrang_diff}), expand the left hand side of \eqref{eq:inst_lagrang_diff} using the definition of \eqref{eq:stochastic_approx}, further utilizing the notation of  $S^t(\bbf_t)$ stated in \eqref{eq:main_prob} and $\bbG^t(\cdot)$ in \eqref{notation}, we can write 
\begin{align}
&S^t(\bbf_t) + \ip{\bbmu,\bbG^t(\bbf_t) + \nu\one} - \frac{\delta\nu}{2}\norm{\bbmu}^2\nonumber\\
&\qquad\quad-S^t(\bbf)-\ip{\bbmu_t,\bbG^t(\bbf) + \nu\one} + \frac{\delta\nu}{2}\norm{\bbmu_t}^2\nonumber
\end{align}
\begin{align}\label{eq:converproof1}
&\leq \frac{1}{2\eta}\Delta_t + \frac{\eta}{2}\big( 2\|{\nabla}_{\bbf}\hat{\ccalL}_{t}(\bbf_{t},\bbmu_{t})\|_{\ccalH}^2+\|\nabla_{\bbmu}\hat{\ccalL}_{t}(\bbf_t,\mathbf{\bbmu}_t)\|^2\big)\nonumber \\
&\quad+\frac{\sqrt{V}\eps}{\eta}\|\bbf_t-\bbf\|_{\ccalH}+\frac{V\eps^2}{\eta}.
\end{align}
where $\Delta_t=(\|\bbf_t-\bbf\|_{\ccalH}^2-\|\bbf_{t+1}-\bbf\|_{\ccalH}^2+\|\bbmu_{t}-\bbmu\|^2-\|\bbmu_{t+1}-\bbmu\|^2)$. Next, let us take the total expectation on both sides of \eqref{eq:converproof1}. From Lemma 
\ref{lemma:bound_primal_dual_grad}, substituting the upper bounds of $\|{\nabla}_{\bbf}\hat{\ccalL}_{t}(\bbf_{t},\bbmu_{t})\|_{\ccalH}^2$ and $\|\nabla_{\bbmu}\hat{\ccalL}_{t}(\bbf_t,\mathbf{\bbmu}_t)\|^2$ to \eqref{eq:converproof1}, the right hand side of \eqref{eq:converproof1} can be written as 
%
\begin{align}\label{inter}
&\mbE\Big[\frac{1}{2\eta}\Delta_t +\frac{\sqrt{V}\eps}{\eta}\|\bbf_t-\bbf\|_{\ccalH}+\frac{V\eps^2}{\eta}\Big]\nonumber
\\
&\quad+\mbE\Big[\frac{\eta}{2}\big(2(4V X^2 C^2 + 4V X^2 L_h^2 E \|\bbmu_t\|^2+2V\lambda^2 \cdot R_{\ccalB}^2)\nonumber
\\
&\quad\quad\quad + E\Big((2K_1+2L_h^2X^2\cdot R_{\ccalB}^2)+ 2\delta^2\eta^2\|\bbmu_t\|^2 \Big)\Big]. 
\end{align}
%
Since each individual $f_{i,t}$ and $f_i$ for $\in\{i,\dots,V\}$ in the ball $\ccalB$ have finite Hilbert norm and is bounded by $R_{\ccalB}$, the term $\|\bbf_t-\bbf\|_{\ccalH}$ can be upper bounded by $2\sqrt{V}R_{\ccalB}$. Next, we define $K\coloneqq 8V X^2 C^2 +4V\lambda^2 \cdot R_{\ccalB}^2+2EK_1+2EL_h^2X^2\cdot R_{\ccalB}^2$. Now using the the bound of  $\|\bbf_t-\bbf\|_{\ccalH}$ and the definition of $K$, we can upper bound the expression in \eqref{inter}, and then collectively writing the left and right hand side terms together, we get
\begin{align}\label{eq:converproof3}
&\mbE\Big[{S(\bbf_t)}-{S(\bbf)}+ \ip{\bbmu,{\bbG(\bbf_t)} + \nu\one} - \frac{\delta\eta}{2}\norm{\bbmu}^2 \\
&\qquad -\ip{\bbmu_t,{\bbG(\bbf)} + \nu\one}\Big] \nonumber\\
& \leq  \mbE\Big[\frac{1}{2\eta}\Delta_t +{\frac{{2}V\eps}{\eta}}R_{\ccalB}+\frac{V\eps^2}{\eta}\Big]+\mbE\Big[\frac{\eta}{2}\big(K+C(\delta)\|\bbmu_t\|^2\big)\Big]. \nonumber
\end{align}
where $C(\delta):=8V X^2 L_h^2 E  + 2E\delta^2\eta^2 -\delta$. Next, we select the constant parameter $\delta$ such that $C(\delta)\leq 0$, which then allows us to drop the term involving $\|\bbmu_t\|^2$ from the second expected term of right-hand side of \eqref{eq:converproof3}. Further, take the sum of the expression in   \eqref{eq:converproof3} over times $t=1,\dots,T$, assume the initialization $\bbf_1=0\in\ccalH^V$ and $\bbmu_1=0\in\mbR_+^E$,  we obtain
\begin{align}
\sum_{t=1}^T&\mbE\Big[{S(\bbf_t)}-{S(\bbf)}+ \ip{\bbmu,{\bbG(\bbf_t)} + \nu\one} - \frac{\delta\eta}{2}\norm{\bbmu}^2 \label{eq:converproof5}\\
&\qquad -\ip{\bbmu_t,{\bbG(\bbf)} + \nu\one}\Big]\nonumber\\
&\leq \frac{1}{2\eta}\big[\|\bbf\|_{\ccalH}^2+\|\bbmu\|^2\big]+ \frac{{2}V\eps T R_{\ccalB}+V\eps^2 T}{\eta}+\frac{\eta K T}{2}.\nonumber
\end{align}
  where we drop the negative terms remaining after the telescopic sum since $\|\bbf_{T+1}-\bbf\|_{\ccalH}^2$ and $\|\bbmu_{T+1}-\bbmu\|^2$ are always positive. It can be observed from \eqref{eq:converproof5} that the right-hand side of this inequality is deterministic. We now take $\bbf$ to be the solution $\bbf_\nu^\star$ of \eqref{eq:prob_zero_cons}, which in turn implies $\bbf_\nu^\star$ must satisfy the inequality constraint of  \eqref{eq:prob_zero_cons}. This means that $\bbf_\nu^\star$ is a feasible point, such that $\sum_{t=1}^T\sum_{(i,j)\in\ccalE}\mu_{ij,t}(g_{ij}(f^\star_{i}(\bbx_{i,t}),f^\star_{j}(\bbx_{i,t}))+\nu)\leq 0$ holds. Thus we can simply drop this term in \eqref{eq:converproof5} and collecting the terms containing $\|\bbmu\|^2$ together, we obtain
\begin{align}\label{eq:converproof6}
&\sum_{t=1}^T\mbE\Big[{S(\bbf_t)}-{S(\bbf_\nu^\star)}+\ip{\bbmu,{\bbG(\bbf_t)}+\nu\one}\Big] - z(\eta,T)\|\bbmu\|^{2}\nonumber\\
&\qquad\leq \frac{1}{2\eta}\|\bbf_\nu^\star\|_{\ccalH}^2+ {\frac{V\eps T}{\eta}}({2}R_{\ccalB}+\eps)+\frac{\eta K T}{2}.
\end{align} 
%
where $z(\eta,T):=\frac{\delta\eta T}{2}+\frac{1}{2\eta}$. Next, we maximize the left-hand side of \eqref{eq:converproof6} over $\bbmu$ to obtain the optimal Lagrange multiplier which controls the growth of the long-term constraint violation, whose closed-form expression is given by
\begin{align}\label{eq:bar_mu}
\bar{\mu}_{ij}= \frac{1}{2(\delta\eta T+1/ \eta)}\big[\sum_{t=1}^T\mbE [g_{ij}(f_t(\bbx_t))]+\nu\big]_+.
\end{align}
Therefore, selecting $\bbmu=\bar{\bbmu}$ in \eqref{eq:converproof6}, we obtain 
\begin{align}\label{eq:converproof8}
&\sum_{t=1}^T\mbE\big[{S(\bbf_t)}-{S(\bbf_\nu^\star)}\big]+\!\!\!\!\sum_{(i,j)\in\ccalE}\!\!\!\!\frac{\Big[\sum_{t=1}^T\Big(\mbE[g_{ij}(f_t(\bbx_t))]+\nu\Big)\Big]_+^2}{2(\delta\eta T+1/ \eta)}\nonumber\\
&\qquad\qquad\leq \frac{1}{2\eta}\|\bbf_\nu^\star\|_{\ccalH}^2+ {\frac{V\eps T}{\eta}}({2}R_{\ccalB}+\eps)+\frac{\eta K T}{2}.
\end{align}
%
Firstly, consider the objective error sequence $\mbE\big[{S(\bbf_t)}-{S(\bbf_\nu^\star)}\big]$, we observe from \eqref{eq:converproof8} that the second term  present on the left-side of the inequality can be dropped without affecting the inequality owing to the fact that it is positive. So we obtain
%
%
\begin{align}
\sum_{t=1}^T\mbE\big[{S(\bbf_t)}-{S(\bbf_\nu^\star)}\big]
\leq \frac{\|\bbf_\nu^\star\|_{\ccalH}^2}{2\eta}+ {\frac{V\eps T}{\eta}}({2}R_{\ccalB}+\eps)+\frac{\eta K T}{2}.\nonumber
\end{align}
Using Lemma \ref{thm:bound_gap} and summing over $t=1, \dots, T$, we have 
\begin{align}\label{eq:converproof9_1}
\sum_{t=1}^T \big[{S(\bbf_\nu^*)}-{S(\bbf^*)}\big]\le \frac{{4}VR_{\ccalB}(C X+\lambda R_{\ccalB})}{\xi}\nu T.
\end{align}
Adding these inequalities, and then taking the average over $T$, we obtain
\begin{align}\label{eq:converproof9_1_temp}
\frac{1}{T}\sum_{t=1}^T\mbE\big[{S(\bbf_t)}-{S(\bbf^*)}\big]&\leq \frac{1}{2\eta T}\|\bbf_\nu^\star\|_{\ccalH}^2+ {\frac{V\eps}{\eta}}({2}R_{\ccalB}+\eps)\nonumber\\
&\hspace{-1cm} +\frac{\eta K}{2} + \frac{{4}VR_{\ccalB}(C X+\lambda R_{\ccalB})}{\xi}\nu. 
\end{align}
%
%
Since $\eps\leq {2}R_{\ccalB}$, we upper bound  $\eps$ present in the second term of the left hand side of \eqref{eq:converproof9_1_temp} by ${2}R_{\ccalB}$. Next, using the definition $\alpha=\eps/\eta$ from Thm. \ref{thm:bound_memory_order}, and setting step size $\eta=1/\sqrt{T}$ and $\nu=\zeta T^{-1/2} + \Lambda \alpha$ for some $(\zeta, \Lambda)>0$, we obtain 
%
%
%
\begin{align}\label{eq:converproof9_1_final}
\frac{1}{T}\sum_{t=1}^T\mbE\big[{S(\bbf_t)}-{S(\bbf^*)}\big]\leq\mathcal{O}(T^{-1/2}+\alpha)
\end{align}
which is the required result in \eqref{eq:func_order}.

%
Next, we establish the bound on the growth of the constraint violation. For this, we first denote  $\ccalL^s$ as the standard Lagrangian  for \eqref{eq:prob_zero_cons} and write it for  ${\bbf}$ and $\bbmu$ as, $\ccalL^s({\bbf},\bbmu) = {S(\bbf)} + \ip{\bbmu,{\bbG(\bbf)}+\nu\one}$. The standard Lagrangian for ${(\bbf_t,\bbmu)}$ and $(\bbf,\bbmu_t)$ are defined similarly. Now using the expressions of $\ccalL^s(\bbf_t,\bbmu)$ and  $\ccalL^s(\bbf,\bbmu_t)$, we  rewrite \eqref{eq:converproof5} as
\begin{align}\label{eq:converproof10_4}
&\sum_{t=1}^T \mbE\big[\ccalL^s(\bbf_t,\bbmu)- \ccalL^s(\bbf,\bbmu_t)\big]\le \frac{1}{2\eta}\mbE\big[\|\bbf\|_{\ccalH}^2+\|\bbmu\|^2\big]\nonumber\\
&\quad\quad+ \frac{{2}V\eps T R_{\ccalB}+V\eps^2 T}{\eta}+\frac{\eta K T}{2} +  \frac{\delta\eta T}{2}\mbE\|\bbmu\|^{2}.
\end{align}
{Since $(\bbf_\nu^*,\bbmu_\nu^*)$ is the optimal pair for standard Lagrangian $\ccalL^s(\bbf,\bbmu)$ of \eqref{eq:prob_zero_cons} and assuming $\mathbf{1}_i$ to be a vector of all zeros except the $i$-th entry which is unity, write for $\bbmu=\mathbf{1}_i+\bbmu_\nu^*$:
\begin{align}\label{eq:converproof10_5}
&\mbE\big[\ccalL^s(\bbf_t,\mathbf{1}_i+\bbmu_\nu^*)\big]-\mbE\big[\ccalL^s(\bbf_\nu^*,\bbmu_\nu^*)\big]\nonumber\\
&=\mbE\big[{S(\bbf_t)}+ \langle \mathbf{1}_i+\bbmu_\nu^*, \bbG(\bbf_t)+ \nu\mathbf{1}\rangle\big]-\mbE\big[\ccalL^s(\bbf_\nu^*,\bbmu_\nu^*)\big]\nonumber\\
&=\mbE\big[{S(\bbf_t)}+ \langle \bbmu_\nu^*, \bbG(\bbf_t)+ \nu\mathbf{1}\rangle\big]-\mbE\big[\ccalL^s(\bbf_\nu^*,\bbmu_\nu^*)\big]\nonumber\\
&\quad + \mbE\big[\langle \mathbf{1}_i, \bbG(\bbf_t)+ \nu\mathbf{1}\rangle\big]\nonumber\\
&=\mbE\big[\ccalL^s(\bbf_t,\bbmu_\nu^*)-\ccalL^s(\bbf_\nu^*,\bbmu_\nu^*)\big] + \mbE\big[G_i(\bbf_t)+\nu\big].
\end{align}
The first equality in \eqref{eq:converproof10_5} is written by using the definition of standard Lagrangian $\ccalL^s$ and $\bbG$ denotes the stacking of the constraints of all edges as defined in the paragraph following  \eqref{notation}. In the second equality, we split $\langle \mathbf{1}_i+\bbmu_\nu^*, \bbG(\bbf_t)+ \nu\mathbf{1}\rangle$ into $\langle \bbmu_\nu^*, \bbG(\bbf_t)+ \nu\mathbf{1}\rangle$ and $\langle \mathbf{1}_i, \bbG(\bbf_t)+ \nu\mathbf{1}\rangle$ using additivity of the inner product. Via the definition of $\ccalL^s$, we rewrite ${S}(\bbf_t)+ \langle \bbmu_\nu^*, \bbG(\bbf_t)+ \nu\mathbf{1}\rangle$ in the second equality as $\ccalL^s(\bbf_t,\bbmu_\nu^*)$ in the third equality. The last term in the third equality is written using the definition of  $\mathbf{1}_i$ and $G_i$ denotes the $i$-th constraint of the stacked constraint vector $\bbG$.
Since $(\bbf_\nu^*,\bbmu_\nu^*)$ is a saddle point of $\ccalL^s$, it holds that
\begin{align}\label{eq:converproof10_6}
\mbE\big[\ccalL^s(\bbf_\nu^*,\bbmu_t)\big]\le \mbE\big[\ccalL^s(\bbf_\nu^*,\bbmu_\nu^*)\big]\le \mbE\big[\ccalL^s(\bbf_t,\bbmu_\nu^*)\big].
\end{align}
From the relation of \eqref{eq:converproof10_6}, we know the first term in the third equality of \eqref{eq:converproof10_5} is positive, and thus drop it to write:
\begin{align}\label{eq:converproof10_7}
&\mbE\big[G_i(\bbf_t)+\nu\big]
\le \mbE\big[\ccalL^s(\bbf_t,\mathbf{1}_i+\bbmu_\nu^*)\big]-\mbE\big[\ccalL^s(\bbf_\nu^*,\bbmu_\nu^*)\big]\nonumber\\
&\quad =\mbE\big[\ccalL^s(\bbf_t,\mathbf{1}_i+\bbmu_\nu^*)\big]-\mbE\big[\ccalL^s(\bbf_\nu^*,\bbmu_t)\big]+\mbE\big[\ccalL^s(\bbf_\nu^*,\bbmu_t)\big]\nonumber\\
&\quad\quad-\mbE\big[\ccalL^s(\bbf_\nu^*,\bbmu_\nu^*)\big]\nonumber\\
& \quad \le \mbE\big[\ccalL^s(\bbf_t,\mathbf{1}_i+\bbmu_\nu^*)\big]-\mbE\big[\ccalL^s(\bbf_\nu^*,\bbmu_t)\big]
\end{align}
where in the second equality we have added and subtracted $\mbE\big[\ccalL^s(\bbf_\nu^*,\bbmu_t)\big]$ and the last inequality comes from the fact that $\mbE\big[\ccalL^s(\bbf_\nu^*,\bbmu_t)\big]-\mbE\big[\ccalL^s(\bbf_\nu^*,\bbmu_\nu^*)\big]\le 0$ from the relation  \eqref{eq:converproof10_6}.
Now, summing \eqref{eq:converproof10_7} over $t=1, \dots, T$ and using \eqref{eq:converproof10_4} we get:
%
\begin{align}\label{eq:converproof10_7_temp}
&\sum_{t=1}^T\mbE\big[G_i(\bbf_t)+\nu\big] \le \sum_{t=1}^T\big(\mbE\big[\ccalL^s(\bbf_t,\mathbf{1}_i+\bbmu_\nu^*)\big]-\mbE\big[\ccalL^s(\bbf_\nu^*,\bbmu_t)\big]\Big)\nonumber\\
 & \le \frac{1}{2\eta}\Big[\|\bbf_\nu^*\|_{\ccalH}^2+\|\mathbf{1}_i+\bbmu_\nu^*\|^2\Big] + V\alpha T (2R_{\ccalB} + \eps)+\frac{ \eta K T}{2} \nonumber\\
&\qquad  +  \frac{\delta\eta {T}}{2}\|\mathbf{1}_i+\bbmu_\nu^*\|^{2},
\end{align}
where $\alpha=\eps/\eta$ defined in Thm. \ref{thm:bound_memory_order}.
Next, using Assumption \ref{as:fifth},  
the first term in \eqref{eq:converproof10_7_temp} by $R_{\ccalB}^2$,  and the second term on the right hand side of \eqref{eq:converproof10_7_temp} is bounded as  
\begin{align}\label{eq:converproof10_8_1}
\|\mathbf{1}_i+\bbmu_\nu^*\|^2&\le 2 \|\mathbf{1}_i\|^2 + 2 \|\bbmu_\nu^*\|^2
\le 2 + 2\Big(\sum_{i=1}^E \mu_{i,\nu}^*\Big)^2\nonumber\\
&\qquad\quad\le 2 + 2 \Big(\frac{{4}VR_{\ccalB}(C X+\lambda R_{\ccalB})}{\xi}\Big)^2.
\end{align}
In the second inequality of \eqref{eq:converproof10_8_1}, we have used the fact that each $\mu_{i,\nu}^*$ is positive thus allowing us to use the inequality $(a^2+b^2+c^2)\le (a+b+c)^2$ where $a,b$ and $c$ are positive. In the last inequality of \eqref{eq:converproof10_8_1} we have used the upper bound of \eqref{eq:boundgap_9} (in the supplementary) to bound the $\big(\sum_{i=1}^E \mu_{i,\nu}^*\big)^2$ term.
Now setting step size $\eta=1/\sqrt{T}$ and using Assumption \ref{as:fifth} and the upper bound of \eqref{eq:converproof10_8_1}, we get 
\begin{align}\label{eq:converproof10_8_1_1}
&\sum_{t=1}^T\mbE\big[G_i(\bbf_t)\big] \leq \sqrt{T} \Gamma + V\alpha T (2R_{\ccalB} + \eps)-\nu T.
\end{align}
where $\Gamma \coloneqq \frac{1}{2}\bigg[R_{\ccalB}^2 + (1+\delta)\bigg(2 + 2 \Big(\frac{{4}VR_{\ccalB}(C X+\lambda R_{\ccalB})}{\xi}\Big)^{2}\bigg) + K\bigg]$. If  $\eps\leq 2 R_{\ccalB}$, then it follows from \eqref{eq:converproof10_8_1_1} that
\begin{align}\label{eq:converproof10_8_1_2}
\frac{1}{T}\sum_{t=1}^T\mbE\big[G_i(\bbf_t)\big] \leq \frac{1}{\sqrt{T}}\Gamma + 4VR_{\ccalB}\alpha -\nu.
\end{align}
Setting $\nu=\zeta T^{-1/2} + \Lambda \alpha$, where $\zeta\geq \Gamma$ and $\Lambda \geq 4VR_{\ccalB}$ ensures the aggregation of constraints gets satisfied on long run, i.e., $\sum_{t=1}^T\mbE\Big[G_i(\bbf_t)\Big]\le 0$, as stated in \eqref{eq:constr_order}. Analogous logic applies for all edges $i\in \ccalE$.\hfill $\blacksquare$
}

\vspace{-2mm}
\bibliographystyle{IEEEtran}
\bibliography{IEEEabrv,bibliography}
\newpage 

\onecolumn
\begin{center} { \Large Supplementary Material for:\vspace{2mm} \\ 
 Adaptive Kernel Learning in Heterogeneous Networks}\vspace{2mm} \\
 by Hrusikesha Pradhan, Amrit Singh Bedi, Alec Koppel, and Ketan Rajawat
\end{center}
\section{Proof of Corollary \ref{thm:representer} }\label{proof_representthm}
The proof generalizes that of the classical Representer Theorem. The inner minimization in \eqref{eq:primaldualprob_emp} with respect to $\bbf$ can be written as
\begin{align}\label{eq:rep_proof1}
\ccalE(\bbf;\ccalS,\bbmu)&=\sum_{i\in\ccalV}\frac{1}{N}\sum_{k=1}^N\bigg[ \ell_i(f_i\big(\bbx_{i,k}), y_{i,k}\big)+\sum_{j\in n_i}  \mu_{ij}\Big(h_{ij}(f_i(\bbx_{i,k}),f_j(\bbx_{i,k}))-\gamma_{ij}\Big) \bigg].
\end{align}
Let the subspace of functions spanned by the kernel function $\kappa(\bbx_{i,t},.)$ for $\bbx_{i,k}\in \ccalS_i$ be denoted as $\ccalF_{\kappa,\ccalS_i}$, i.e.,
\begin{align}
\ccalF_{\kappa,\ccalS_i}=\text{span}\{\kappa(\bbx_{i,k}):1\le k\leq N\}.
\end{align}
We denote the projection of $f_i$ on the subspace $\ccalF_{\kappa,\ccalS_i}$ as $f_{ip}$ and the component perpendicular to the subspace as $f_{i\perp}$, which can be written as $f_{i\perp}=f_i-f_{ip}$. Now we can write 
\begin{align}\label{eq:rep_proof2}
f_i(\bbx_{ik})=\langle f_i,\kappa(\bbx_{ik},.)\rangle&=\langle f_{ip},\kappa(\bbx_{i,k},.)\rangle + \langle f_{i\perp},\kappa(\bbx_{i,k},.)\rangle\nonumber\\
&=\langle f_{ip},\kappa(\bbx_{i,k},.)\rangle=f_{ip}(\bbx_{i,k}).
\end{align}
Thus the evaluation of $f_i$ at any arbitrary training  point $\bbx_{ik}$ is independent of $f_{i\perp}$. Using this fact, we can now write \eqref{eq:rep_proof1} as,
%
\begin{align}\label{eq:rep_proof3}
\ccalE(\bbf;\ccalS,\bbmu)=& \sum_{i\in\ccalV}\frac{1}{N}\sum_{k=1}^N\bigg[ \ell_i(f_{ip}\big(\bbx_{i,k}), y_{i,k}\big)+\sum_{j\in n_i}  \mu_{ij}\Big(h_{ij}(f_{ip}(\bbx_{i,k}),f_j(\bbx_{i,k}))-\gamma_{ij}\Big) \bigg].
\end{align}
Thus from \eqref{eq:rep_proof3}, we can say that $\ccalE(\bbf;\ccalS,\bbmu)$ is independent of $f_{i\perp}$. As  we are minimizing \eqref{eq:empirical_lagrangian} with respect to $f_i$, the evaluation of $f_j$ at the training point of node $i$ can be treated as a constant in  $\ccalE(\bbf;\ccalS,\bbmu)$ which is the first part in  \eqref{eq:empirical_lagrangian}. Additionally, note that $\lambda\cdot\|f_i\|_{\ccalH}^2\cdot 2^{-1}\geq \lambda\cdot\|f_{ip}\|_{\ccalH}^2\cdot 2^{-1}$. Therefore, given any $\bbmu$, the quantity $\ccalE(\bbf;\ccalS,\bbmu)+\sum_{i=1}^V\lambda\cdot\|f_i\|_{\ccalH}^2\cdot 2^{-1}$ is minimized at some $f_i^*(\bbmu_i)$ such that $f_i^*(\bbmu_i)$ lies in $\ccalF_{k,\ccalS_i}$. This holds specifically for $\bbmu_i^*$ where $f_i^*=f_i^*(\bbmu_i^*)$, there by completing the proof. \hfill $\blacksquare$

%


\section{Statement and Proof of Lemma  \ref{thm:bound_gap}}
\label{app:proof_bound_gap}
Using Assumption \ref{as:fourth}, we bound the gap between optimal of problem \eqref{eq:main_prob}
and \eqref{eq:prob_zero_cons} and  is presented as Lemma \ref{thm:bound_gap}. 
\begin{lemma}\label{thm:bound_gap}
Under Assumption \ref{as:second}, \ref{as:fourth} and \ref{as:fifth}, for $0\le \nu\le\xi/2$, it holds that:
\begin{align}
S(\bbf_\nu^*)-S(\bbf^*)\le \frac{{4}VR_{\ccalB}(C X+\lambda R_{\ccalB})}{\xi}\nu.
\end{align}
\end{lemma}
%

\begin{proof}
Let $(\bbf^*,\bbmu^*)$ be the solution to \eqref{eq:main_prob} and  $(\bbf_{\nu}^*,\bbmu_{\nu}^*)$ be the solution to \eqref{eq:prob_zero_cons}. As $\nu \le \frac{\xi}{2} \le\xi$, there exists a strictly feasible primal solution $\bbf^{\dagger}$ such that  $G(\bbf^{\dagger})+\mathbf{1}\nu\le G(\bbf^{\dagger})+\mathbf{1}\xi$, where $\mathbf{1}$ denotes the vector of all ones and $G$ denotes the stacked vector of constraints as defined in the proof of Theorem \ref{thm:convergence}. Hence strong duality holds for \eqref{eq:prob_zero_cons}. Therefore using the definition of $S(\bbf)$ from {\eqref{eq:kernel_stoch_opt_global}}, we have 
\begin{align}
S(\bbf_\nu^*)&=\min_{\bbf} S(\bbf) + \langle \mu_{\nu}^*, G(\bbf)+ \mathbf{1}\nu \rangle\nonumber\\
& \le S(\bbf^*) + \langle \mu_{\nu}^*, G(\bbf^*)+ \mathbf{1}\nu \rangle\label{eq:boundgap_1}\\
& \le S(\bbf^*) + \nu \langle \mu_{\nu}^*, \mathbf{1}\rangle\label{eq:boundgap_2}
\end{align}
where the inequality in \eqref{eq:boundgap_1} comes from from the optimality of $\bbf_\nu^*$ and \eqref{eq:boundgap_2} comes from the fact that $G(\bbf^*)\le 0$. Next using Assumption \ref{as:fourth}, we have strict feasibility of $\bbf^\dagger$, so using \eqref{eq:boundgap_1} we can write:
 \begin{align}
 S(\bbf_\nu^*)& \le S(\bbf^\dagger) + \langle \mu_{\nu}^*, G(\bbf^\dagger)+ \mathbf{1}\nu \rangle\nonumber\\
 & =S(\bbf^\dagger) + \langle \mu_{\nu}^*, G(\bbf^\dagger)+ \mathbf{1}(\nu + \xi-\xi) \rangle\nonumber\\
 &=S(\bbf^\dagger) + \langle \mu_{\nu}^*, G(\bbf^\dagger)+ \mathbf{1} \xi\rangle + \langle \mu_{\nu}^*, \mathbf{1}(\nu -\xi) \rangle \nonumber\\
 &\le S(\bbf^\dagger) +(\nu -\xi) \langle \mu_{\nu}^*,\mathbf{1}\rangle.\label{eq:boundgap_3}
 \end{align}
 Thus from \eqref{eq:boundgap_3}, we can equivalently write,
 \begin{align}\label{eq:boundgap_3_1}
 \langle \mu_{\nu}^*,\mathbf{1}\rangle\le \frac{S(\bbf^\dagger)-S(\bbf_\nu^*)}{\xi-\nu}
\end{align}  
Now we upper bound the difference of $S(\bbf^\dagger)-S(\bbf_\nu^*)$. Using the definition of $S(\bbf)$, we write the difference of $S(\bbf^\dagger)-S(\bbf_\nu^*)$ as
\begin{align}\label{eq:boundgap_4}
S(\bbf^\dagger)\!-\!S(\bbf_\nu^*)=\mbE\sum_{i\in\ccalV}\!\!\big[\ell_i(f^\dagger_{i}\big(\bbx_{i,t}), y_{i,t}\big)\!-\!\ell_i(f_{i,\nu}^\star\big(\bbx_{i,t}), y_{i,t}\big)\!\big]\!+\frac{\lambda}{2}\sum_{i\in\ccalV}\!\Big(\|f^\dagger_{i} \|^2_{\ccalH}- \|f_{i,\nu}^\star \|^2_{\ccalH}\Big).
\end{align}
Next, we bound the  sequence in \eqref{eq:boundgap_4} as
\begin{align}\label{eq:boundgap_5}
|S(\bbf^\dagger)\!-\!S(\bbf_\nu^*)|&\!\leq\! \mbE\!\sum_{i\in\ccalV}\!\big[|\ell_i(f^\dagger_{i}\big(\bbx_{i,t}), y_{i,t}\big)\!-\!\ell_i(f_{i,\nu}^\star\big(\bbx_{i,t}), y_{i,t}\big)|\big]+\frac{\lambda}{2}\!\sum_{i\in\ccalV}\!|\|f^\dagger_{i} \|^2_{\ccalH}- \|f_{i,\nu}^\star \|^2_{\ccalH}|\nonumber\\
&\leq \mbE\sum_{i\in\ccalV} C|f^\dagger_{i}\big(\bbx_{i,t})-f_{i,\nu}^\star\big(\bbx_{i,t})|+\frac{\lambda}{2}\!\sum_{i\in\ccalV}\!|\|f^\dagger_{i} \|^2_{\ccalH}- \|f_{i,\nu}^\star \|^2_{\ccalH}|,
\end{align}
where using triangle inequality we write the first inequality and then using Assumption \eqref{as:second} of Lipschitz-continuity condition we write the second inequality. Further, using reproducing property of $\kappa$ and Cauchy-Schwartz inequality, we simplify  $|f^\dagger_{i}\big(\bbx_{i,t})-f_{i,\nu}^\star\big(\bbx_{i,t})|$ in \eqref{eq:boundgap_5} as
\begin{align}\label{eq:boundgap_6}
|f^\dagger_{i}\big(\bbx_{i,t})-f_{i,\nu}^\star\big(\bbx_{i,t})|&=|\langle f^\dagger_{i}-f_{i,\nu}^\star,\kappa(\bbx_{i,t},\cdot)\rangle| \leq \|f^\dagger_{i}-f_{i,\nu}^\star\|_{\ccalH}\cdot \|\kappa(\bbx_{i,t},\cdot)\|_{\ccalH}\leq {2}R_\ccalB X
\end{align} 
where the last inequality comes from Assumption \ref{as:first} and \ref{as:fifth}. 
Now, we consider the $|\|f^\dagger_{i} \|^2_{\ccalH}- \|f_{i,\nu}^\star \|^2_{\ccalH}|$ present in the right-hand side of \eqref{eq:boundgap_5},
\begin{align}\label{eq:boundgap_7}
&|\|f^\dagger_{i} \|^2_{\ccalH}- \|f_{i,\nu}^\star \|^2_{\ccalH}|\leq \|f^\dagger_{i}-f_{i,\nu}^\star\|_{\ccalH}\cdot \|f^\dagger_{i}+f_{i,\nu}^\star\|_{\ccalH}\!\leq\! 4R_{\ccalB}^2.
\end{align}
Substituting \eqref{eq:boundgap_6} and \eqref{eq:boundgap_7} in \eqref{eq:boundgap_5}, we obtain
\begin{align}\label{eq:boundgap_8}
|S(\bbf^\dagger)-S(\bbf_\nu^*)| \leq {2}VCR_\ccalB X+{2}V\lambda R_{\ccalB}^2
={2}VR_{\ccalB}(C X+\lambda R_{\ccalB}).
\end{align}
Now using \eqref{eq:boundgap_8}, we rewrite \eqref{eq:boundgap_3_1} as
\begin{align}\label{eq:boundgap_9}
 \langle \mu_{\nu}^*,\mathbf{1}\rangle\le \frac{S(\bbf^\dagger)-S(\bbf_\nu^*)}{\xi-\nu}\le \frac{{2}VR_{\ccalB}(C X+\lambda R_{\ccalB})}{\xi-\nu}\le \frac{{4}VR_{\ccalB}(C X+\lambda R_{\ccalB})}{\xi}.
\end{align}
Finally, we use \eqref{eq:boundgap_9} in \eqref{eq:boundgap_2} and get the required result:
\begin{align}
S(\bbf_\nu^*)-S(\bbf^*)\le \frac{{4}VR_{\ccalB}(C X+\lambda R_{\ccalB})}{\xi}\nu.
\end{align}
\end{proof}
The importance of Lemma \ref{thm:bound_gap} is that it establishes the fact that the gap between the solutions of the problem \eqref{eq:main_prob}
and \eqref{eq:prob_zero_cons} is $\ccalO(\nu)$.

\section{Statement and Proof of Lemma \ref{lemma:bound_primal_dual_grad}}
\label{app:bound_primal_dual_gradient}

We  bound the primal and dual stochastic gradients used for \eqref{eq:projection_hat} and \eqref{eq:dualupdate_edge}, respectively in the following lemma.
%
\begin{lemma}\label{lemma:bound_primal_dual_grad}
Using Assumptions \ref{as:first}-\ref{as:fifth}, the mean-square-magnitude of the primal and dual gradients of the stochastic augmented Lagrangian $\hat{\ccalL}_t(\bbf,\bbmu)$ defined in \eqref{eq:stochastic_approx} are upper-bounded as
\begin{align}\label{eq:lemma1_final}
\!\!\!\!\mathbb{E}[\| \nabla_\bbf\hat{\ccalL}_t(\bbf,\bbmu)\|^2_{\ccalH}]&\leq  4V X^2 C^2 + 4V X^2 L_h^2 {E} \|\bbmu\|^2+2V \lambda^2 R_{\ccalB}^2\\
\!\!\!\!\mathbb{E}\Big[\| \nabla_{\bbmu}\hat{\ccalL}_t(\bbf,\bbmu)\|^2_{\ccalH}\Big]\!\! &\leq  {E}\Big(\!(2K_1\!\!+\!2L_h^2X^2 R_{\ccalB}^2)\!+\! 2\delta^2\eta^2\|\bbmu\|^2\! \Big)\label{eq:lemma1_2_final}
\end{align}
for some $0<K_1<\infty$.
\end{lemma}
\begin{proof}
In this proof for any $(\bbf,\bbmu) \in \ccalH^V\times \mbR^{E}_+ $ we upper bound the mean-square-magnitude of primal gradient as
\begin{align}\label{eq:lemma1_1}
\mathbb{E}[\| \nabla_\bbf\hat{\ccalL}_t(\bbf,\bbmu)\|^2_{\ccalH}]&=\mathbb{E}[\|vec(\nabla_{fi}\hat{\ccalL}_t(\bbf,\bbmu))\|^2_{\ccalH}]\le V \max_{i\in \ccalV}\mathbb{E}[\|\nabla_{fi}\hat{\ccalL}_t(\bbf,\bbmu)\|^2_{\ccalH}],
\end{align}
%
where for the first equality we have used the fact that the functional gradient is a concatenation of functional gradients associated with each agent. The second inequality is obtained by considering the worst case estimate across the network. In the right-hand side of \eqref{eq:lemma1_1} we substitute the value of $\nabla_{fi}\hat{\ccalL}_t(\bbf,\bbmu)$ from \eqref{eq:lagg_derv} to obtain,
\begin{align}\label{eq:delta_f_bound1_temp}
\mathbb{E}[\| \nabla_\bbf\hat{\ccalL}_t(\bbf,\bbmu)\|^2_{\ccalH}]
&\le V \!\max_{i\in \ccalV}\mathbb{E}\big[\|\!\big[\ell_i'(f_{i}(\bbx_{i,t}),y_{i,t})\!+\!\!\sum_{j\in n_i}\!\mu_{ij}h'_{ij}(f_{i}(\bbx_{i,t}),\!f_{j}(\bbx_{i,t}))\big]\! \kappa(\bbx_{i,t},\cdot)\!+\!\lambda f_{i}\|^2_{\ccalH}\big]\\
&\le V \max_{i\in \ccalV} \mathbb{E}\big[2\|\big[\ell_i'(f_{i}(\bbx_{i,t}),y_{i,t})+\sum_{j\in n_i}\mu_{ij}h'_{ij}(f_{i}(\bbx_{i,t}),f_{j}(\bbx_{i,t}))\big] \kappa(\bbx_{i,t},\cdot)\|_{\ccalH}^2\big]\!+2V\lambda^2 \|f_{i}\|_{\ccalH}^2.\nonumber
\end{align}
In \eqref{eq:delta_f_bound1_temp}, we have used the fact that $\|a+b\|_{\ccalH}^2\le 2\cdot(\|a\|_{\ccalH}^2+\|b\|_{\ccalH}^2)$ for any $a, b \in \ccalH$, i.e., the sum of squares inequality. Next we again use the sum of squares inequality for the first bracketed term in the right hand side of \eqref{eq:delta_f_bound1_temp} and also used Assumption \ref{as:fifth} to upper bound $\|f_{i}\|_{\ccalH}^2$  by $R_\ccalB^2$ and get,
\begin{align}\label{eq:delta_f_bound1}
\mathbb{E}[\| \nabla_\bbf\hat{\ccalL}_t(\bbf,\bbmu)\|^2_{\ccalH}]\le V \max_{i\in \ccalV} \mathbb{E}\big[4\|\ell_i'(f_{i}(\bbx_{i,t}),y_{i,t})\kappa(\bbx_{i,t},\cdot)\|^2+4\|\sum_{j\in n_i}\mu_{ij}h'_{ij}(f_{i}(\bbx_{i,t}),f_{j}(\bbx_{i,t})) \kappa(\bbx_{i,t},\cdot)\|_{\ccalH}^2\big]+c(\lambda),
\end{align}
where $c(\lambda):=\!\!2V\!\lambda^2\! \cdot\! R_{\ccalB}^2$. Using Cauchy-Schwartz inequality, the first term on the right-hand side of \eqref{eq:delta_f_bound1} can be written as $$\|\ell_i'(f_{i}(\bbx_{i,t}),y_{i,t})\kappa(\bbx_{i,t},\cdot)\|^2\le \|\ell_i'(f_{i}(\bbx_{i,t}),y_{i,t})\|^2 \|\kappa(\bbx_{i,t},\cdot)\|^2.$$ Then using Assumptions \ref{as:first} and \ref{as:second}, we bound $\|\ell_i'(f_{i}(\bbx_{i,t}),y_{i,t})\|^2 $ by $C^2$ and  $\|\kappa(\bbx_{i,t},\cdot)\|^2$ by $X^2$. Similarly we use Cauchy-Schwartz inequality for the second term in \eqref{eq:delta_f_bound1} and bound $\|\kappa(\bbx_{i,t},\cdot)\|^2$ by $X^2$. Now using these, \eqref{eq:delta_f_bound1} can be written as,
\begin{align}\label{eq:lemma1_final_temp1}
\mathbb{E}[\| \nabla_\bbf\hat{\ccalL}_t(\bbf,\bbmu)\|^2_{\ccalH}]\le \!4V' C^2 + 4V' \|\sum_{j\in n_i}\mu_{ij}h'_{ij}(f_{i}(\bbx_{i,t}),f_{j}(\bbx_{i,t})) \|_{\ccalH}^2+c(\lambda),
\end{align}
where $V':=VX^2$. Using Assumption \ref{as:third}, we bound $h'_{ij}(f_{i}(\bbx_{i,t}),f_{j}(\bbx_{i,t}))$ present in the second term on the right-hand side of \eqref{eq:lemma1_final_temp1} by $L_h$ and then taking the constant $L_h$ out of the summation, we get
\begin{align}\label{eq:lemma1_final_temp2}
\mathbb{E}[\| \nabla_\bbf\hat{\ccalL}_t(\bbf,\bbmu)\|^2_{\ccalH}]\le 4V' C^2+ 4V'L_h^2 \|\sum_{j=1}^{ |n_i|}\mu_{ij} \|^2+c(\lambda).
\end{align}
Here, $|n_i|$  denotes the number of neighborhood nodes of agent $i$. Then we have used the fact $\|\sum_{j=1}^{ |n_i|}\mu_{ij} \|^2\le {|n_i|} \sum_{j=1}^{ |n_i|}|\mu_{ij}|^2$ and got
\begin{align}\label{eq:lemma1_final_temp3}
\mathbb{E}[\| \nabla_\bbf\hat{\ccalL}_t(\bbf,\bbmu)\|^2_{\ccalH}]\le 4V' C^2+ 4V' L_h^2 {|n_i|} \sum_{j=1}^{ |n_i|}|\mu_{ij}|^2+c(\lambda).
\end{align}
Next we upper bound ${ |n_i|}$ and $\sum_{j=1}^{ |n_i|}|\mu_{ij}|^2$ by ${E}$ and $\|\bbmu\|^2$ and write \eqref{eq:lemma1_final_temp3} as 
%
\begin{align}\label{eq:lemma1_final}
\mathbb{E}[\| \nabla_\bbf\hat{\ccalL}_t(\bbf,\bbmu)\|^2_{\ccalH}]\le 4V' C^2 + 4V' L_h^2 {E} \|\bbmu\|^2+c(\lambda).
\end{align}
Thus \eqref{eq:lemma1_final} which establishes an the upper bound on $\mathbb{E}[\| \nabla_\bbf\hat{L}_t(\bbf,\bbmu)\|^2_{\ccalH}]$ is valid. 
 
With this in hand, we now shift focus to deriving a similar upper-bound on the magnitude of the dual stochastic gradient of the Lagrangian $\mathbb{E}\Big[\| \nabla_{\bbmu}\hat{\ccalL}_t(\bbf,\bbmu)\|^2_{\ccalH}\Big]$ as 
\begin{align}\label{eq:delta_mu_temp1}
\mathbb{E}\Big[\| \nabla_{\bbmu}\hat{\ccalL}_t(\bbf,\bbmu)\|^2_{\ccalH}\Big]
&=\mathbb{E}\|\text{vec}(h_{ij}(f_{i}(\bbx_{i,t}),f_{j}(\bbx_{i,t}))-\gamma_{ij}+\nu-\delta\eta\mu_{ij})\|_{\ccalH}^2\nonumber\\
&\le {E} \max_{(i,j)\in\ccalE}\mathbb{E}\|h_{ij}(f_{i}(\bbx_{i,t}),f_{j}(\bbx_{i,t}))-\gamma_{ij}+\nu-\delta\eta\mu_{ij}\|_{\ccalH}^2\nonumber\\
&\le {E} \max_{(i,j)\in\ccalE}\mathbb{E}\|h_{ij}(f_{i}(\bbx_{i,t}),f_{j}(\bbx_{i,t}))+\nu-\delta\eta\mu_{ij}\|_{\ccalH}^2.
\end{align}
In the first equality we write the concatenated version of the dual stochastic gradient associated with each agent, whereas the second inequality is obtained by considering the worst case bound. In the third inequality, we use the fact $|a-b-c|^2 \le |a-c|^2$ owing to the fact that the right hand side of the inequality is a scalar. Next, applying $\|a+b\|_{\ccalH}^2\le 2\cdot(\|a\|_{\ccalH}^2+\|b\|_{\ccalH}^2)$ for any $a, b \in \ccalH$, we get 
\begin{align}\label{eq:delta_mu_temp2}
\mathbb{E}\Big[\| \nabla_{\bbmu}\hat{\ccalL}_t(\bbf,\bbmu)\|^2_{\ccalH}\Big]&\le {E}\big(2\mathbb{E}\|h_{ij}(f_{i}(\bbx_{i,t}),f_{j}(\bbx_{i,t}))+\nu\|^2_{\ccalH}+ 2\delta^2\eta^2|\mu_{ij}|^2\big).
\end{align}
Here we have ignored the $\nu^2$ term as $\nu < 1$ and can be subsumed within the first term.
Then we bound the first term in \eqref{eq:delta_mu_temp2} using Assumption \ref{as:third} and the second term is upper bounded by $\|\bbmu\|^2$
\begin{align}\label{eq:delta_mu_temp3}
&\mathbb{E}\Big[\| \nabla_{\bbmu}\hat{\ccalL}_t(\bbf,\bbmu)\|^2_{\ccalH}\Big]\le {E}\Big(2\big(K_1+L_h^2\mathbb{E}(|f_{i}(\bbx_{i,t})|^2)\big)+ 2\delta^2\eta^2\|\bbmu\|^2 \Big).
\end{align}
Next, we use $|f_{i}(\bbx_{i,t})|^2=|\langle f_{i},\kappa(\bbx_{i,t},\cdot)\rangle_\ccalH|^2\le \|f_{i}\|_\ccalH^2\cdot\|\kappa(\bbx_{i,t},\cdot)\|_\ccalH^2$ and then we have upper bounded $\|f_{i}\|_\ccalH^2$ and $\|\kappa(\bbx_{i,t},\cdot)\|_\ccalH^2$ by $R_{\ccalB}^2$ and $X^2$, and we obtain
\begin{align}\label{eq:lemma1_2_final}
\mathbb{E}&\Big[\| \nabla_{\bbmu}\hat{\ccalL}_t(\bbf,\bbmu)\|^2_{\ccalH}\Big]\le {E}((2K_1+2L_h^2X^2\cdot R_{\ccalB}^2)+ 2\delta^2\eta^2\|\bbmu\|^2 ).
\end{align}
%
\end{proof}

\vspace{-4mm}

\section{Statement and Proof of Lemma \ref{lemma:diff_of_grad}}
\label{app:bound_grad_diff_func_proj_func}
The following lemma bounds the difference of projected stochastic functional gradient and un-projected stochastic functional gradient.
  \vspace{-1mm}
 \begin{lemma}\label{lemma:diff_of_grad}
The difference between the stochastic functional gradient defined by ${\nabla}_{\bbf}\hat{\ccalL}_{t}(\bbf_{t},\bbmu_{t})$ and projected stochastic functional gradient $\tilde{\nabla}_{\bbf}\hat{\ccalL}_{t}(\bbf_{t},\bbmu_{t})$, is bounded as
 \begin{align}\label{eq:diff_of_grad}
 \|{\nabla}_{\bbf}\hat{\ccalL}_{t}(\bbf_{t},\bbmu_{t})-\tilde{\nabla}_{\bbf}\hat{\ccalL}_{t}(\bbf_{t},\bbmu_{t})\|_{\ccalH}\le \frac{\sqrt{V}\eps}{\eta}
 \end{align}
 for all $t>0$. Here, $\eta>0$ is the algorithm step-size and $\eps>0$ is the error tolerance parameter of the KOMP.
 \end{lemma}
%
\begin{proof}
 Considering the squared-Hilbert- norm difference of the left hand side of \eqref{eq:diff_of_grad} 
 \begin{subequations}
 \begin{align}
\|{\nabla}_{\bbf}\hat{\ccalL}_{t}(\bbf_{t},\bbmu_{t})-\tilde{\nabla}_{\bbf}\hat{\ccalL}_{t}(\bbf_{t},\bbmu_{t})\|_{\ccalH}^2
&=\frac{1}{\eta^2}\|\eta{\nabla}_{\bbf}\hat{\ccalL}_{t}(\bbf_{t},\bbmu_{t})-\eta\tilde{\nabla}_{\bbf}\hat{\ccalL}_{t}(\bbf_{t},\bbmu_{t})\|_{\ccalH}^2\label{eq:proof_diff1}\\
 &=\frac{1}{\eta^2}\|\eta{\nabla}_{\bbf}\hat{\ccalL}_{t}(\bbf_{t},\bbmu_{t})+\bbf_{t+1}-\bbf_t\|^2\label{eq:proof_diff2}.
 \end{align}
 In \eqref{eq:proof_diff2}, we used \eqref{eq:projected_func_update} for the second term on the right hand side of \eqref{eq:proof_diff1}. we re-arrange the terms in  \eqref{eq:proof_diff2} and  then, we use $\bbf_t-\eta\nabla_{\bbf}\hat{\ccalL}_{t}(\bbf_{t},\bbmu_{t})$, which can easily by identified as $\tilde{\bbf}_{t+1}$ given in \eqref{eq:stacked_sgd_tilde} and obtain  
 \begin{align}
 \|{\nabla}_{\bbf}\hat{\ccalL}_{t}(\bbf_{t},\bbmu_{t})-\tilde{\nabla}_{\bbf}\hat{\ccalL}_{t}(\bbf_{t},\bbmu_{t})\|_{\ccalH}^2
 &=\frac{1}{\eta^2}\|\bbf_{t+1}-\big(\bbf_t-\eta{\nabla}_{\bbf}\hat{\ccalL}_{t}(\bbf_{t},\bbmu_{t})\big)\|^2 \nonumber\\
 &=\frac{1}{\eta^2}\|\bbf_{t+1}-\tilde{\bbf}_{t+1}\|^2\label{eq:proof_diff3}\\
 &=\frac{1}{\eta^2}\sum_{i=1}^V\|f_{i,t+1}-\tilde{f}_{i,t+1}\|^2\le\frac{1}{\eta^2}V\eps^2\label{eq:proof_diff4}.
\end{align}  
\end{subequations}
 In \eqref{eq:proof_diff3} we used the stacked version of $\tilde{f}_{i,t+1}$ to substitute $\tilde{\bbf}_{t+1}$ in place of $\bbf_t-\eta\tilde{\nabla}_{\bbf}\hat{\ccalL}_{t}(\bbf_{t},\bbmu_{t})$. In \eqref{eq:proof_diff4} we used the error tolerance parameter of the KOMP update. Then taking the square root of \eqref{eq:proof_diff4} gives the inequality stated in \eqref{eq:diff_of_grad} and concludes the proof .
 \end{proof}
\section{Definition and Proof of Lemma \ref{lemma:inst_lagrang_diff}}
\label{app:bound_inst_lag_diff}
Next, Lemma \ref{lemma:inst_lagrang_diff}  characterizes the instantaneous Lagrangian difference $\hat{\ccalL}_{t}(\bbf_t,\bbmu)-\hat{\ccalL}_{t}(\bbf,\bbmu_t)$.
 %
 \begin{lemma}\label{lemma:inst_lagrang_diff}
 Under Assumptions \ref{as:first}-\ref{as:fifth} and the primal and dual updates generated from Algorithm \ref{alg:soldd}, the instantaneous Lagrangian difference satisfies the following decrement property 
 \begin{align}\label{eq:inst_lagrang_diff}
&\hat{\ccalL}_t(\bbf_t,\bbmu)-\hat{\ccalL}_t(\bbf,\bbmu_t)\nonumber\\
&\leq \frac{1}{2\eta}\!\big(\|\bbf_t\!-\!\bbf\|_{\ccalH}^2-\|\bbf_{t+1}\!\!-\!\!\bbf\|_{\ccalH}^2+\|\bbmu_{t}\!-\!\bbmu\|^2-\|\bbmu_{t+1}\!-\!\bbmu\|^2\big)\nonumber\\
&\quad+\! \frac{\eta}{2}\big( 2\|{\nabla}_{\bbf}\hat{\ccalL}_{t}(\bbf_{t},\bbmu_{t})\|_{\ccalH}^2+\|\nabla_{\bbmu}\hat{L}_{t}(\bbf_t,\bbmu_t)\|^2\big)+\frac{\sqrt{V}\eps}{\eta}\|\bbf_t-\bbf\|_{\ccalH}+\frac{V\eps^2}{\eta}. 
 \end{align}
 \end{lemma}
 \begin{proof}
 Considering the squared hilbert norm of the difference between the iterate $\bbf_{t+1}$ and any feasible point $\bbf$ with each individual $f_i$ in the ball $\ccalB$ and exoanding it using the \eqref{eq:projected_func_update}, we get
 \begin{align}\label{eq:diff_ft+1_f_1}
 \|\bbf_{t+1}-\bbf\|_{\ccalH}^2=\|\bbf_t-\eta \tilde{\nabla}_{\bbf}\hat{\ccalL}_{t}(\bbf_{t},\bbmu_{t})-\bbf\|_{\ccalH}^2
 &=\!\|\bbf_t\!-\!\bbf\|_{\ccalH}^2-2\eta\langle \bbf_t-\bbf,\tilde{\nabla}_{\bbf}\hat{\ccalL}_{t}(\bbf_{t},\bbmu_{t})\rangle +\eta^2\|\tilde{\nabla}_{\bbf}\hat{\ccalL}_{t}(\bbf_{t},\bbmu_{t})\|_{\ccalH}^2\nonumber\\
&=\|\bbf_t-\bbf\|_{\ccalH}^2+2\eta\langle \bbf_t-\bbf,{\nabla}_{\bbf}\hat{\ccalL}_{t}(\bbf_{t},\bbmu_{t})-\tilde{\nabla}_{\bbf}\hat{\ccalL}_{t}(\bbf_{t},\bbmu_{t})\rangle\nonumber\\
 \!\!&\quad\!-2\eta\langle \bbf_t-\bbf,{\nabla}_{\bbf}\hat{\ccalL}_{t}(\bbf_{t},\bbmu_{t})\rangle+\eta^2\|\tilde{\nabla}_{\bbf}\hat{\ccalL}_{t}(\bbf_{t},\bbmu_{t})\|_{\ccalH}^2
 \end{align}
 where we have added and subtracted $2\eta\langle \bbf_t-\bbf,{\nabla}_{\bbf}\hat{\ccalL}_{t}(\bbf_{t},\bbmu_{t})\rangle$ and gathered like terms on the right-hand side. Now to handle the second term on the right hand side of  \eqref{eq:diff_ft+1_f_1}, we use Cauchy Schwartz inequality along with the Lemma \ref{lemma:diff_of_grad} to replace the directional error associated with sparse projections with the functional difference defined by the KOMP stopping criterion:
 \begin{align}\label{diff_ft+1_f_2}
 &\langle \bbf_t-\bbf,{\nabla}_{\bbf}\hat{\ccalL}_{t}(\bbf_{t},\bbmu_{t})-\tilde{\nabla}_{\bbf}\hat{\ccalL}_{t}(\bbf_{t},\bbmu_{t})\rangle\le \|\bbf_t-\bbf\|_{\ccalH}\|{\nabla}_{\bbf}\hat{\ccalL}_{t}(\bbf_{t},\bbmu_{t}\!)-\tilde{\nabla}_{\bbf}\hat{\ccalL}_{t}(\bbf_{t},\bbmu_{t}\!)\|_{\ccalH}\le\frac{\sqrt{V}\eps}{\eta}\|\bbf_t-\bbf\|_{\ccalH}.
 \end{align}
 Now to bound the norm of $\tilde{\nabla}_{\bbf}\hat{\ccalL}_{t}(\bbf_{t},\bbmu_{t})$, the last term in the right hand side of  \eqref{eq:diff_ft+1_f_1}, we add and subtract ${\nabla}_{\bbf}\hat{\ccalL}_{t}(\bbf_{t},\bbmu_{t})$ and then use the identity $\|a+b\|_{\ccalH}^2\le 2.(\|a\|_{\ccalH}^2+\|b\|_{\ccalH}^2)$ and further use Lemma \ref{lemma:diff_of_grad} and finally get,
 \begin{align}\label{diff_ft+1_f_3}
 \|\tilde{\nabla}_{\bbf}\hat{\ccalL}_{t}(\bbf_{t},\bbmu_{t})\|_{\ccalH}^2\le 2\frac{V\eps^2}{\eta^2}+2\|{\nabla}_{\bbf}\hat{\ccalL}_{t}(\bbf_{t},\bbmu_{t})\|_{\ccalH}^2.
 \end{align}
 Now we substitute the expressions in \eqref{diff_ft+1_f_2} and \eqref{diff_ft+1_f_3} in for the second and fourth terms in \eqref{eq:diff_ft+1_f_1} which allows us to write
 \begin{align}
  \|\bbf_{t+1}-\bbf\|_{\ccalH}^2\le\|\bbf_t-\bbf\|_{\ccalH}^2+2\sqrt{V}\eps\|\bbf_t\!-\!\bbf\|_{\ccalH}-2\eta\langle \bbf_t-\bbf,{\nabla}_{\bbf}\hat{\ccalL}_{t}(\bbf_{t},\bbmu_{t})\rangle\!+\!2V\eps^2\!+\!2\eta^2\|{\nabla}_{\bbf}\hat{\ccalL}_{t}(\bbf_{t},\bbmu_{t})\|_{\ccalH}^2.
 \end{align}
 By re-ordering the terms of the above equation, we get
 \begin{align}\label{diff_ft+1_f_4}
 \langle \bbf_t-\bbf,{\nabla}_{\bbf}\hat{\ccalL}_{t}(\bbf_{t},\bbmu_{t})\rangle
  \!&\le\! \frac{1}{2\eta}\big(\|\bbf_t-\bbf\|_{\ccalH}^2-\|\bbf_{t+1}-\bbf\|_{\ccalH}^2\big)\!+\!\frac{\sqrt{V}\eps}{\eta}\|\bbf_t-\bbf\|_{\ccalH}+\frac{V\eps^2}{\eta}+\eta\|{\nabla}_{\bbf}\hat{\ccalL}_{t}(\bbf_{t},\bbmu_{t})\|_{\ccalH}^2.
 \end{align}
 Using the first order convexity condition for instantaneous Lagrangian $\hat{\ccalL}_{t}(\bbf_{t},\bbmu_{t})$, since it is convex with respect to $\bbf_t$ and write
 \begin{equation}\label{eq:1storderconv}
\!\! \hat{\ccalL}_{t}(\bbf_{t},\bbmu_{t})-\hat{\ccalL}_{t}(\bbf,\bbmu_{t})\le \langle \bbf_t-\bbf,{\nabla}_{\bbf}\hat{\ccalL}_{t}(\bbf_{t},\bbmu_{t})\rangle.
 \end{equation}
 Next we use \eqref{eq:1storderconv} in \eqref{diff_ft+1_f_4} and get
 \begin{align}\label{diff_ft+1_f_final}
 &\hat{\ccalL}_{t}(\bbf_{t},\bbmu_{t})-\hat{\ccalL}_{t}(\bbf,\bbmu_{t})\le\! \frac{1}{2\eta}\big(\|\bbf_t-\bbf\|_{\ccalH}^2\!-\!\|\bbf_{t+1}-\bbf\|_{\ccalH}^2\big)\!+\!\frac{\sqrt{V}\eps}{\eta}\|\bbf_t-\bbf\|_{\ccalH}+\frac{V\eps^2}{\eta}+\eta\|{\nabla}_{\bbf}\hat{\ccalL}_{t}(\bbf_{t},\bbmu_{t})\|_{\ccalH}^2.
 \end{align}
 Similarly, we consider the squared difference of dual variable update $\bbmu_{t+1}$ in \eqref{eq:dualupdate} and an arbitrary dual variable $\bbmu$, 
 \begin{align}\label{eq:mu_t+1_mu1}
 \|\bbmu_{t+1}-\bbmu\|^2&=\|[\bbmu_t+\eta\nabla_{\bbmu}\hat{\ccalL}_{t}(\bbf_t,\mathbf{\bbmu}_t)]_{+}-\bbmu\|^2\le \|\bbmu_t+\eta\nabla_{\bbmu}\hat{\ccalL}_{t}(\bbf_t,\mathbf{\bbmu}_t)-\bbmu\|^2.
  \end{align}
  The above inequality in \eqref{eq:mu_t+1_mu1} comes from the non-expansiveness of the projection operator $[.]_+$. Next we expand the square of the right-hand side of \eqref{eq:mu_t+1_mu1} and get,
  \begin{align}\label{eq:mu_t+1_mu2}
  \!\! \|\bbmu_{t+1}-\bbmu\|^2&\le  \|\bbmu_{t}-\bbmu\|^2 + 2\eta\nabla_{\bbmu}\hat{\ccalL}_{t}(\bbf_t,\bbmu_t)^T(\bbmu_t-\bbmu)+\eta^2\|\nabla_{\bbmu}\hat{\ccalL}_{t}(\bbf_t,\bbmu_t)\|^2.
  \end{align}
 We re-arrange the terms in the above expression and get,
 \begin{align}\label{eq:mu_t+1_mu3}
\!\! \nabla_{\bbmu}\hat{\ccalL}_{t}(\bbf_t,\mathbf{\bbmu}_t)^T(\bbmu_t-\bbmu) &\!\!\geq \frac{1}{2\eta}\big(\|\bbmu_{t+1}\!-\!\bbmu\|^2\!-\!\|\bbmu_{t}\!-\!\bbmu\|^2\big) -\frac{\eta}{2}\|\nabla_{\bbmu}\hat{\ccalL}_{t}(\bbf_t,\mathbf{\bbmu}_t)\|^2.
 \end{align}
 Since the instantaneous Lagrangian $\hat{\ccalL}_{t}(\bbf_{t},\bbmu_{t})$ is concave with respect to the dual variable $\bbmu_t$, i.e., 
 \begin{align}\label{eq:mu_concave}
 \hat{\ccalL}_{t}(\bbf_{t},\bbmu_{t})-\hat{\ccalL}_{t}(\bbf_{t},\bbmu)\geq \nabla_{\bbmu}\hat{\ccalL}_{t}(\bbf_t,\bbmu_t)^T(\bbmu_t-\bbmu).
 \end{align}
 Next we use the left-hand side of the inequality \eqref{eq:mu_concave} in \eqref{eq:mu_t+1_mu3} and get the expression,
 \vspace{-0.3cm}
 \begin{align}\label{eq:mu_t+1_mu4}
\hat{\ccalL}_{t}(\bbf_{t},\bbmu_{t})-\hat{\ccalL}_{t}(\bbf_{t},\bbmu)&\geq \frac{1}{2\eta}\big(\|\bbmu_{t+1}-\bbmu\|^2-\|\bbmu_{t}-\bbmu\|^2\big) -\frac{\eta}{2}\|\nabla_{\bbmu}\hat{\ccalL}_{t}(\bbf_t,\mathbf{\bbmu}_t)\|^2.
 \end{align}
We subtract \eqref{eq:mu_t+1_mu4} from \eqref{diff_ft+1_f_final} to obtain   the final expression,
 \begin{align}
 \hat{\ccalL}_{t}(\bbf_{t},\bbmu)-\hat{\ccalL}_{t}(\bbf,\bbmu_t)&\leq \frac{1}{2\eta}\big(\|\bbf_t\!-\!\bbf\|_{\ccalH}^2-\|\bbf_{t+1}-\bbf\|_{\ccalH}^2+\|\bbmu_{t}\!-\!\bbmu\|^2\!\!-\!\|\bbmu_{t+1}\!-\!\bbmu\|^2\big) +\frac{\eta}{2}\big( 2\|{\nabla}_{\bbf}\hat{\ccalL}_{t}(\!\bbf_{t},\bbmu_{t}\!)\|_{\ccalH}^2+\|\nabla_{\bbmu}\hat{\ccalL}_{t}(\!\bbf_t,\mathbf{\bbmu}_t\!)\|^2\big)\nonumber\\
&\quad+\frac{\sqrt{V}\eps}{\eta}\|\bbf_t-\bbf\|_{\ccalH}+\frac{V\eps^2}{\eta}.
 \end{align} 
  \end{proof}
\section{ Definition and proof of Lemma \ref{lemma:dist_subspace}}\label{lemma:dist_function}
In this section, we present the proof of Theorem \ref{thm:bound_memory_order}, where we upper bound the growth of the dictionary. But before going into the proof of Theorem \ref{thm:bound_memory_order}, we present Lemma \ref{lemma:dist_subspace} which defines the notion of measuring distance of a point from subspace which will be subsequently used in the proof of Theorem \ref{thm:bound_memory_order}.
Using Lemma \ref{lemma:dist_subspace}, we establish the relation between the stopping criteria of the compression procedure to a Hilbert subspace distance.
\begin{lemma}\label{lemma:dist_subspace}
Define the distance of an arbitrary feature vector $\bbx$ obtained by the feature transformation $\phi(\bbx)=\kappa(\bbx, \cdot)$  to the subspace of the Hilbert space
spanned by a dictionary $\bbD$ of size $M$, i.e., $\ccalH_{\bbD}$ as 
\begin{align}
\text{dist}(\kappa(\bbx, \cdot), \ccalH_{\bbD})=\min_{\bbf \in \ccalH_{\bbD}} \|\kappa(\bbx, \cdot)-\bbv^T\boldsymbol{\kappa}_{\bbD}(\cdot)\|_\ccalH.
\end{align}
This set distance gets simplified to the following least-squares projection when dictionary, $\bbD \in \reals^{p\times M}$ is fixed
\begin{align}\label{eq:dist1}
\text{dist}(\kappa(\bbx, \cdot), \ccalH_{\bbD})=\|\kappa(\bbx, \cdot)- [\bbK_{\bbD,\bbD}^{-1}\boldsymbol{\kappa}_{\bbD}(\bbx)]^{T}\boldsymbol{\kappa}_{\bbD}(\cdot)\|_\ccalH.
\end{align}
\end{lemma}
\begin{proof}
 The distance to the subspace $\ccalH_{\bbD}$ is defined as
\begin{align}\label{eq:dist_eq1}
\text{dist}(\kappa(\bbx, \cdot), \ccalH_{\bbD})= \min_{\bbf \in \ccalH_{\bbD}}\|\kappa(\bbx, \cdot)-\bbv^T\boldsymbol{\kappa}_{\bbD}(\cdot)\|_\ccalH=\min_{\bbv \in  \reals^{M}} \|\kappa(\bbx, \cdot)-\bbv^T\boldsymbol{\kappa}_{\bbD}(\cdot)\|_\ccalH
\end{align} 
 where the second equality comes from the fact that as $\bbD$ is fixed so minimizing over $\bbf$ translates down to minimizing over $\bbv$   since it is the only free parameter now. Now we solve \eqref{eq:dist_eq1} and obtain $\bbv^*=\bbK_{\bbD,\bbD}^{-1}\boldsymbol{\kappa}_{\bbD}(\bbx)$ minimizing \eqref{eq:dist_eq1} in a manner similar to logic which yields \eqref{eq:hatparam_update}. Now using $\bbv^*$ we obtain the required result given in \eqref{eq:dist1}, thereby concluding the proof.
\end{proof}
\section{Proof of Corollary \ref{thm:iter_comp}}\label{proof_iter_comp}
Considering the optimality gap bound in \eqref{eq:converproof9_1_temp} and using the fact that $\eps\leq {2}R_{\ccalB}$, we can write \eqref{eq:converproof9_1_temp} with step size $\eta=1/\sqrt{T}$ as
\begin{align}\label{eq:iter_comp_1}
\frac{1}{T}\sum_{t=1}^T\mbE\big[S(\bbf_t)-S(\bbf^*)\big]&\leq\frac{1}{2\eta \sqrt{T}}\|\bbf_\nu^\star\|_{\ccalH}^2+ {\frac{V\eps}{\eta}}{4}R_{\ccalB}+\frac{ K}{2\sqrt{T}} + \frac{{4}VR_{\ccalB}(C X+\lambda R_{\ccalB})}{\xi}\nu.
\end{align}
Now denoting $Q\coloneqq 4VR_{\ccalB}(CX+\lambda R_{\ccalB})/\xi$ and using $\nu=\zeta T^{-1/2} + \Lambda \alpha$, we write \eqref{eq:iter_comp_1} as
\begin{align}\label{eq:iter_comp_2}
\frac{1}{T}\sum_{t=1}^T\mbE\big[S(\bbf_t)-S(\bbf^*)\big]&\leq\frac{1}{2\eta \sqrt{T}}\|\bbf_\nu^\star\|_{\ccalH}^2+ {\frac{V\eps}{\eta}}{4}R_{\ccalB}+\frac{ K}{2\sqrt{T}} + Q (\zeta T^{-1/2} + \Lambda \alpha)\nonumber\\
& = \frac{1}{\sqrt{T}}\Big(\frac{\|\bbf_\nu^\star\|_{\ccalH}^2}{2}+\frac{K}{2}+ Q\zeta\Big) + \alpha(4VR_{\ccalB}+Q\Lambda).
\end{align}
Now for the optimality gap in \eqref{eq:iter_comp_2} to be less than $\varepsilon$, it is sufficient if we consider that both the terms in the right hand side of the inequality \eqref{eq:iter_comp_2} are bounded by $\varepsilon/2$, i.e.,
\begin{align}
\alpha(4VR_{\ccalB}+Q\Lambda)&\leq \varepsilon/2\label{eq:iter_comp_3}\\
\frac{1}{\sqrt{T}}\Big(\frac{\|\bbf_\nu^\star\|_{\ccalH}^2}{2}+\frac{K}{2}+ Q\zeta\Big) &\leq \varepsilon/2\label{eq:iter_comp_4}
\end{align}
Thus, from \eqref{eq:iter_comp_3} and \eqref{eq:iter_comp_4} we can deduce that for an optimality gap to be less than $\varepsilon$ we require $\ccalO(1/\varepsilon^2)$ iterations with $\alpha$ satisfying \eqref{eq:iter_comp_3}. Now using the bound of $\alpha$ from \eqref{eq:iter_comp_3} in \eqref{eq:agent_mo} of Theorem \ref{thm:bound_memory_order}, we get model order complexity bound of $\ccalO(1/\varepsilon^{2p})$ required to achieve an optimality gap of $\varepsilon$.

\newpage

   \end{document}